\documentclass[abstracton,paper=a4,headings=normal,bibliography=toc,index=totoc,
	]{scrartcl}
	
\usepackage[utf8]{inputenc}
\usepackage{lmodern}
\usepackage[T1]{fontenc}
\usepackage[german,english]{babel}

\ifpdfoutput{\usepackage{microtype}}{}

\usepackage{fixltx2e}  

\usepackage{amsmath,amsfonts,amssymb,amsthm}

\usepackage{textcomp}
\usepackage{exscale} 

\usepackage{cancel}   
\usepackage{graphicx} 
\usepackage{url} 
\usepackage[arrow, matrix, curve]{xy}
\usepackage{xspace}
 
\usepackage{soul}

\usepackage[makeindex]{splitidx}
\newindex[Index of notation]{not}
\newindex[Index]{idx}

\usepackage[usenames,dvipsnames]{xcolor}

\ifpdfoutput{
\usepackage[unicode,pdftex,hypertexnames=false,plainpages=false,pdfpagelabels=true,setpagesize=false]{hyperref}\hypersetup{breaklinks=true,linktocpage=false,colorlinks=true,urlcolor=black,linkcolor=black,citecolor=black,pdfauthor={Julian Hofrichter, Tat Dat Tran, Jürgen Jost},pdftitle={The uniqueness of hierarchically extended backward solutions of the Wright–Fisher model}
}}{}

\newcommand{\bel}[1]{\begin{equation}\label{#1}}

\newcommand{\be}{\begin{equation}}
\newcommand{\ba}{\begin{eqnarray}}
\newcommand{\ea}{\end{eqnarray}}
\newcommand{\rf}[1]{(\ref{#1})}
\newcommand{\bi}{\bibitem}
\newcommand{\qe}{\end{equation}}

\newcommand{\R}{\mathbb{R}}

\newcommand{\N}{\mathbb{N}}
\newcommand{\Np}{\mathbb{N}_+}
\newcommand{\Z}{\mathbb{Z}}

\newcommand{\D}{\Delta}
\newcommand{\cube}{\square}
\newcommand{\cubex}{\boxtimes}
\newcommand{\map}{\longrightarrow}
\newcommand{\too}{\longrightarrow}

\newcommand{\half}{\frac12}

\newcommand{\smin}{\setminus}

\providecommand{\ce}{\mathrel{\mathop:}=}
\providecommand{\ec}{=\mathrel{\mathop:}}
\providecommand{\ceq}{\mathrel{\mathop:}\Leftrightarrow}
\providecommand{\co}{\colon}

\providecommand{\txind}[1]{#1\index{#1}}

\providecommand{\cl}[1]{\overline{#1}}

\providecommand{\maxind}{\hat{\imath}}

\providecommand{\com}[1]{}

\providecommand{\abs}[1]{\lvert#1\rvert}

\providecommand{\set}[1]{\lbrace#1\rbrace}

\providecommand{\bigset}[1]{\big\lbrace#1\big\rbrace}
\providecommand{\Bigset}[1]{\Big\lbrace#1\Big\rbrace}

\DeclareMathOperator{\argmax}{arg\,max}

\newcommand{\bd}{\partial}
\newcommand{\dd}[2]{\frac{\partial#1}{\partial#2}}
\newcommand{\ddsq}[2]{\frac{\partial^2#1}{{(\partial#2)}^2}}

\newcommand{\ddd}[3]{\frac{\partial^2#1}{\partial{#2}\partial{#3}}}

\newcommand{\si}{r} 				
\newcommand{\ri}{s}				
\newcommand{\m}{m}				

\providecommand{\fdt}{\,\cdot\,}

\newcommand{\ind}{\chi\mathnormal{}}

\newcommand{\leb}{\lambda\hspace{-5 pt}\lambda}

\renewcommand{\phi}{\varphi}

\newenvironment{eqn*}{\begin{equation*}}{\end{equation*}}

\newcommand{\ie}{i.\,e.\@\xspace}
\newcommand{\eg}{e.\,g.\@\xspace}
\newcommand{\cf}{cf.\@\xspace}

\newcommand{\wlg}{w.\,l.\,o.\,g.\@\xspace}

\newcommand{\resp}{resp.\@\xspace}
\newcommand{\zeroth}{0{th}\xspace}

\newcommand{\second}{2{nd}\xspace}

\newcommand{\ord}{-{th}\xspace}

\providecommand{\WF}{Wright--Fisher\xspace}

\providecommand{\KBE}{Kolmogorov backward equation\xspace}
\providecommand{\KFE}{Kolmogorov forward equation\xspace}

\setkomafont{captionlabel}{\small\itshape}

\numberwithin{equation}{section} 
\swapnumbers
 
\theoremstyle{plain} 
\newtheorem{thm}{Theorem}[section]
\newtheorem{prop}[thm]{Proposition}
\newtheorem{lem}[thm]{Lemma}
\newtheorem{cor}[thm]{Corollary}

\theoremstyle{definition}
\newtheorem{dfi}[thm]{Definition}

\newtheorem{rmk}[thm]{Remark}

\allowdisplaybreaks[1]

\definecolor{gruen}{rgb}{0.6,0,0.5}

\renewcommand{\L}{\mathcal{L}}

\begin{document}

\setcounter{section}{0}

\title{The uniqueness of hierarchically extended backward solutions of the Wright–Fisher model}

\author{Julian Hofrichter, Tat Dat Tran, Jürgen Jost}

\date{\today}

\maketitle


\begin{abstract}
The diffusion approximation of the Wright-Fisher model of population genetics leads to partial differentiable equations, the so-called Kolmogorov equations,  with an operator that degenerates at the boundary. Standard tools do not apply, and in fact, solutions lack regularity properties. In this paper, we develop a regularising blow-up scheme for a certain class of solutions of the backward Kolmogorov equation, the iteratively extended global solutions presented in \cite{THJ5}, and establish their uniqueness. As the model describes the random genetic drift of several alleles at the same locus from a backward perspective, the  singularities result from the loss of an allele. While in an analytical approach, this causes substantial difficulties, from a biological or geometric perspective, this is a natural process that can be analyzed in detail. The presented scheme regularises the solution via a tailored successive transformation of the domain.
\end{abstract}


\section{Introduction}

The Wright-Fisher model \cite{fisher,wright1} is concerned with  genetic drift which constitutes the most basic mechanism of mathematical population genetics. Let us briefly describe this. In a finite population of fixed size, parents are randomly sampled and pass the alleles which they are carrying on to the offspring generation. By repeating this process over many (non-overlapping) generations, the model describes the evolution of the probabilities of the different alleles in the population. In the basic setting, the model covers a single locus only. Extensions to several loci are possible, as is the inclusion of  mutation, selection, or a spatial population structure. This has the leading research strand in mathematical population genetics (\cite{ewens,buerger}) since the  work of Kimura \cite{kimura1,kimura2,kimura3}. 

Even for the  original basic model, there  remain interesting and deep mathematical questions, particularly when one passes to the  diffusion approximation. Following Kimura, one shifts to a model with an infinite population size and continuous time. This leads to  the forward and backward Kolmogorov equations. The forward equation is a partial differential equation of parabolic type, while the backward equation is not parabolic as it evolves backward in time; it is the adjoint of the former w. r. t. a suitable product. Both equations become degenerate at the boundary.  This makes standard PDE theory inapplicable, and this therefore is a source of mathematical challenges. 

In this paper, we study  solutions of the \KBE 
\bel{comp0}
-\dd{}{t} u(p,t) = \half\sum_{i,j=1}^n p^i(\delta^i_j-p^j)\ddd{}{p^i}{p^j}u(p,t)\ec L_n^\ast u(p,t).
\qe
$p^i$ is the relative frequency of allele $i$; note that $p^0$ does not appear in \rf{comp0} because of the normalization $\sum_{i=0}^n p^i=1$. If one of the frequencies $p^i$ becomes 0, the corresponding coefficient also become 0. We define  the \KBE in the 
closure of the probability simplex $\Delta_n=\{(p^1,\dots ,p^n): p^i > 0, \sum_{j=1}^n p^j < 1\}$. 

Thus,  the differential operator in \eqref{eq_back_n_ext} becomes degenerate at the boundary of $\Delta_n$. (The fact that \rf{comp0} is not parabolic because time is running backward is not such a serious problem, because of the structure of the model and the duality with the -- parabolic -- \KFE .)

The Kolmogorov  equations have been studied with tools both from the  theory of stochastic processes, see for instance \cite{ethier1,ethier2,ethier3,karlin}, and  from the theory of  partial differential equations  \cite{epstein1,epstein2}. 
These  approaches are quite general and cab produce  existence, uniqueness and regularity results, but cannot come up with explicit formulas, for instance for the expected time of loss of an allele. 

Therefore, people have also looked more closely into the specific and explicit structure of the model. Of particular relevance is  the global aspect. This means that one wants to  connect the solutions in the interior of the simplex and on its boundary faces. Over the years,  several global representation formulas have been derived. For a survey, see  Section 5.10 of \cite{ewens} and  \cite{buerger}, but in order to set the stage, we want to discuss certain results in more detail and with a different focus. 

It was observed in \cite{Sato1978} that one can  write the Kolmogorov backward operator in the form
\begin{align}\label{comp1}
 \Lambda_n^\ast u(x,t) \ce \half\sum_{i,j=0}^n x^i(\delta^i_j-x^j)\ddd{}{x^i}{x^j}u(x,t),
\end{align}
using the variables $(x^0,x^1,\dots ,x^n)$ with $\sum_{j=0}^n x^j=1$  in place of $L_n^\ast u(p,t)$ (\cf equation~\eqref{comp0}) 
with $(p^1,\dots ,p^n)$ and $p^0=1-\sum_{i=1}^n p^i$ implicitly determined (for our notation, see Sections \ref{sec_simplex}, in particular \rf{eq_stand_simpl} and \rf{eq_Ln*_def}), that is, one works on the simplex $\{x^0 +x^1 +\dots x^n=1, x^i\ge
0\}$, i.e., the variable $x^0$ is included. This formulation has the advantage of being symmetric w.r.t. all $x^i$, but because the operator invokes more independent variables than the dimension of the underlying space, the elliptic operator becomes degenerate. Here, we have opted to work with $L_n^\ast$, but for the comparison with the literature, we shall utilize the version \rf{comp1}. 

A starting point for much of  the literature that we shall  discuss here  is  the observation of Wright \cite{wright2} that the degeneracy at the boundary may be removed if one includes mutation. Let the mutation rate  $m_{ij}$  be the probability that  the offspring carries the mutant $j$ instead of the parent's allelel $i$; also,  $m_{ii}=-\sum_{j\ne i} m_{ij}$. This produces the  Kolmogorov backward operator
\begin{align}\label{comp2}
 \Lambda_n^\ast u(x,t) \ce
 \half\sum_{i,j=0}^n x^i(\delta^i_j-x^j)\ddd{}{x^i}{x^j}u(x,t) + \sum_{j=0}^n \sum_{i=0}^n m_{ij}x^i \frac{\partial}{\partial x^j}.
\end{align}
Calculations become  simpler, if following Wright \cite{wright2}, 
\bel{comp3}
m_{ij}=\frac{1}{2}\mu_j >0 \text{ for }i\neq j. 
\qe
This means that the mutation rates depends only on the target gene (the factor $\frac{1}{2}$ is inserted solely for purposes of normalization) and are positive.  With \rf{comp3}, \rf{comp2} becomes
\begin{align}\label{comp4}
 \Lambda_n^\ast u(x,t) \ce
 \half\sum_{i,j=0}^nx^i(\delta^i_j-x^j)\ddd{}{x^i}{x^j}u(x,t) + \half \sum_{j=0}^n (\mu_j -\sum_{i=0}^n \mu_i)x^j \frac{\partial}{\partial x^j}.
\end{align}
In this case, one obtains a unique stationary distribution for the Wright--Fisher diffusion, given by the Dirichlet distribution with parameters
$\mu_0,\dots ,\mu_n$. 
A further simplification occurs when
\bel{comp5}
\mu_0 = \dots =\mu_n =:\mu >0. 
\qe 
This means that  all mutation rates are identical. From a biological perspective, the assumption \rf{comp3} that the mutation rates only depend on the target gene is not so natural  (the mutation rate should rather depend on the initial instead of the target gene, but \rf{comp5} remedies that deficit in a certain sense), but for our purposes,  the more crucial issue is the assumption of positivity. 

Several papers have studied this model and derived explicit formulas for the transition density of the process with generator \rf{comp4} including \cite{lit-fack,Shi1977,Gri1979,Gri1980,Shi1981,Tav1984,EG1993,GS2010}; these, however, were rather of a local nature, as they did not connect solutions in the interior and in boundary strata of the domain. A useful tool is   Kingman's coalescent \cite{King1982}, the method of tracing lines of descent back into the past and analyzing their merging patterns (for a brief introduction, see also \cite{jost_bio}). In particular, some of these formulas  likewise extend to the limiting case $\mu=0$ in \rf{comp5}. Ethier--Griffiths \cite{EG1993} have  for the transition density
\bel{lit1}
P(t,x,dy)=\sum_{M \ge 1} d^0_M(t) \sum_{|\alpha|=M, \alpha \in\Z^n_+}{|\alpha| \choose \alpha}x^{\alpha}\mathrm{Dir}(\alpha, dy),
\qe
also for the case $\mu=0$. 
Previously this was known  under the assumption $\mu >0$. Here, $\mathrm{Dir}$ is the Dirichlet distribution, and $ d^0_M(t)$ is the number of equivalence classes of lines of descent of length $M$ at time $t$ in Kingman's coalescent. For the latter,  analytical formulas can be found  in \cite{Tav1984}. \rf{lit1} has been studied further in many subsequent papers, for instance \cite{GS2010}. Shimakura in \cite{Shi1981} came up with the less explicit formula
\ba
\nonumber
P(t,x,dy)&=&\sum_{m\ge 1} e^{-\lambda_m t} E_m(x,dy)\\
\nonumber
 &=& \sum_{K\in \Pi} P(t,x,y) dS_K(y)\\
\label{lit2}
 &=&\sum_{K\in \Pi} e^{-\lambda_m t} E_{m,K}(x,y)dS_K(y).
\ea
Here, the $\lambda_m$ are the eigenvalues introduced above, and $E_m$ stands for the projection onto the corresponding eigenspace, and the index $K$ enumerates the faces of the simplex. However,  the Dirichlet distribution in \rf{lit1} and the measure $dS_K(y)$ in \rf{lit2} both become singular when $y$ approaches the boundary of $K$, which means that the transition from one face into one of its boundary faces becomes singular in this scheme, considering the solutions on the individual faces invoked by the sum. Thus, in fact, \rf{lit2} is simply a decomposition into the various modes of the solutions of a linear PDE, summed over all faces of the simplex; this illustrates the rather local character of the solution scheme. 

\medskip

In the present paper, we want to get a more detailed analytical picture of the behavior at the boundary and investigate global solutions, in particular their uniqueness, on the entire state space including its stratified boundary. In an important recent work, Epstein and Mazzeo \cite{epstein1,epstein2} have developed PDE techniques to tackle the issue of solving PDEs on a manifold with corners that degenerate at the boundary with the same leading terms as the  Kolmogorov backward equation for the Wright-Fisher model~\eqref{comp0}
in the closure of the probability simplex \text{in $(\cl{\D}_n)_{-\infty}=\cl{\D}_n\times(-\infty,0)$}. This analysis is heavily based on the identification of appropriate function spaces. In our context, their spaces $C^{k,\gamma}_{WF}(\cl{\D}_n)$ would consist of $k$ times continuously differentiable functions whose $k$th derivatives are Hölder continuous with exponent $\gamma$ w.r.t. the Fisher metric. (This only holds true for $L^*_n$, although E\&M also use this construction for their generalised setting.) In terms of the Euclidean metric on the simplex, this means that a weaker Hölder exponent (essentially $\frac{\gamma}{2}$) is required in the normal than in the tangential directions at the boundary. Using this framework, they subsequently show that if the initial values are of class  $C^{k,\gamma}_{WF}(\cl{\D}_n)$, then there exists a unique solution in that class. This result is very satisfactory from the perspective of PDE theory (see e.g. \cite{jost_pde}). The solutions we consider here, however,  are not even continuous, let alone of some class $C^{0,\gamma}(\cl{\D}_n)$, as we want to study the boundary transitions.  Likewise, in our setting, the (stationary) uniqueness assertion does not apply, which E\&M have established for  regular (in particular, globally continuous) solutions by a modified version of the Hopf boundary point lemma and some maximum principle (yielding a similar, but more general result as proposition~10.2 in \cite{THJ5}).

This assessment also applies to other works which treat uniqueness issues in the context of degenerate PDEs, but are not adapted to the very specific class of solutions at hand. This includes the extensive work by Feehan \cite{fee} where -- amongst other issues -- the uniqueness of solutions of elliptic PDEs whose differential operator degenerates along a certain portion of the boundary $\bd_0\Omega$ of the domain $\Omega$ is established: For a problem with a partial Dirichlet boundary condition, \ie boundary data are only given $\bd\Omega\smin\bd_0\Omega$,  a so-called second-order boundary condition is applied for the degenerate boundary area; this is that a solution needs to be such that the leading terms of the differential operator continuously vanishes towards $\bd_0\Omega$, while the solution itself is also of class $C^1$ up to $\bd_0\Omega$. Within this framework, Feehan then shows that -- under  certain natural conditions  -- degenerate operators satisfy a corresponding maximum principle for the partial boundary condition, which assures the uniqueness of a solution. Although this in principle may also apply to solutions of \WF diffusion equations, however, this does not entirely cover the situation at hand, as, if $n\geq2$, $L^*$ does only partially degenerate towards the boundary (instances of codimension 1). More precisely, its degeneracy behaviour is stepwise, corresponding to the stratified boundary structure of the domain $\cl{\D}_n$, and hence does not satisfy the requirements for Feehan's scenario. Furthermore, in the language of \cite{fee}, the intersection of the regular and the degenerate boundary part $\bd\bd_0\Omega$, would encompass a hierarchically iterated boundary-degeneracy structure, which is beyond the scope of that work.

Therefore, in this paper, we continue the detailed investigation of the boundary behavior of solutions of the (extended) \KBE \eqref{comp0} started in \cite{THJ5}. In analytical terms, the issue is the regularity of solutions at singularities of the boundary, that is, where two or more faces of the simplex $\Delta_n$ meet.The particular extension paths from the boundary into the interior of the simplex (they have nothing to do, however, with Kingman's coalescent lines of descent as utilized in some of the literature discussed above) may result in boundary singularities at certain strata of the boundary the domain.  We are interested in the directions in which the singularities of the boundary of the simplex are approached from the interior, because our aim is to resolve these boundary singularities.

In contrast to some of the approaches discussed above that invoke strong tools from the theory of stochastic processes, our approach is not stochastic, but analytic and geometric in nature. Analytically, our approach is more related to that of \cite{epstein1,epstein2}. Geometrically, we develop   constructions, within the spirit of information geometry, that is, the geometry of probability distributions, see \cite{amari,ajls}. This will provide us with tools  that on one hand can naturally handle the  general aspects mentioned  above, but on the other hand can still derive explicit formulas. This is part of a general research program, see \cite{Dat,julian,THJ1,THJ2,THJ3,THJ4,THJ5}. 

\medskip

Let us now describe in more specific terms what we achieve in this paper. Based on the previous work \cite{THJ5}, we continue the analysis of solutions of the \textit{(extended) \KBE} for the diffusion approximation of the \WF model
\begin{align}\label{eq_back_n_ext}
\begin{cases}
L^* U(p,t)=-\dd{}{t} U(p,t) 	&\text{in ${\big(\overline{\Delta}_{n}\big)}_{-\infty}=\cl{\D}_n\times(-\infty,0)$}\\%
U (p,0) =f(p)					&\text{in $\overline{\Delta}_{n}$, $f\in {L}^2\big(\bigcup_{k=0}^n\bd_k\D_n\big)$}\\
\end{cases}
\end{align}
for $U(\,\cdot\,,t)\in C_p^2\big(\cl{\D}_n\big)$ 
for each fixed $t\in(-\infty,0)$ and $U(p,\,\cdot\,)\in C^1((-\infty,0))$ for each fixed $p\in\overline{\Delta}_{n}$ 
\resp the \textit{stationary (extended) \KBE}
\begin{equation}\label{eq_back_n_stat_ext}
\begin{cases}
L^* U(p)=0 		&\text{in $\overline{\Delta}_n\smin\bd_0\D_n$}\\
U(p) = f(p)	&\text{in $\bd_0\D_n$}\\
\end{cases}
\end{equation}
for $U\in C_p^2\big(\cl{\D}_n\big)$
and in either case with
\begin{align}\label{eq_Ln*_def}
L^* u(p,t) \ce \half\sum_{i,j=1}^n\big(p^i(\delta^i_j-p^j)\big)\ddd{}{p^i}{p^j}u(p,t).
\end{align}
being the corresponding \textit{backward operator}\sindex[not]{Ln*@$L_n^*$}. 

Emerging solutions of the backward Kolmogorov equation may be interpreted as the probability distribution over ancestral states yielding some given current state of allele frequencies with time running backward as indicated by the name. Such an ancestral state could have possessed more alleles than the current state, as on the path towards that latter state, some alleles that had been originally present in the population could have been lost. In analytical terms, one could assume that such a loss of allele event is continuous, in the sense that the relative frequency of the corresponding allele simply goes to~0. Geometrically, however, this means that the process  from the interior of a probability simplex enters to into some boundary stratum and henceforth stays there. Also, when two or more alleles got lost, they could have disappeared in different orders from the population. A corresponding global and hierarchical solution for the \KBE  that persists and stays regular across different such loss of allele events in the past was constructed in the preceding paper \cite{THJ5}, which was technically rather involved. This not been achieved before in the literature, but is indispensable for a complete understanding and a rigorous solution of the \KBE. This approach is now completed by also establishing the uniqueness for this class of solution in the stationary case, which likewise has not been considered before.  

The key is the degeneracy at the boundary of the Kolmogorov equations. While from an analytical perspective, this presents a profound difficulty for obtaining boundary regularity of the solutions of the equations, from a biological or geometric perspective, this is very natural because it corresponds to the loss of some alleles  from the population in finite time by random drift. And from a stochastic perspective, this has to happen almost surely. For this reason, the above equations are not accessible by standard theory, as perhaps the square root of the coefficients of the second order terms of $L^*$ is not Lipschitz continuous up to the boundary. As a consequence, in particular the uniqueness of solutions to the above {\KBE}s may not be derived from standard results. Instead, such degenerate equations arising from population biology have been anlyzed by Epstein and Mazzeo (\cf \cite{epstein1}, \cite{epstein2}) only recently. Contrasting their aim of developing a preferably general and broad theory, we rather focus on the some very specific aspect within this field, \ie the regularity \resp uniqueness of a certain class of functions, and use a strategy which is specifically adapted to the situation at hand.  These functions are the hierarchically extended solutions of the \KBE\index{Kolmob@\KBE} developed in \cite{THJ5}.

Our strategy here is aimed at gaining global regularity in the closure of the domain by resolving any incompatibilities between different boundary strata. 
That will be achieved by an appropriate transformation of the relevant part of the domain (\ie the simplex $\D_n$, \cf below) which transports the whole problem to the corresponding image domain of a product of a simplex and a cube. Simultaneously, the iteratively extended solutions\index{backward extension!iterated} 
are turned into corresponding solutions of the transformed equation, which are then of sufficient global regularity, in particular are globally continuous. For generic iteratively extended solutions\index{backward extension!iterated}
this does not yet yield a corresponding regularity, however, their transformation image may assumingly be extended that way, which is likewise reasonable in terms of the underlying model.


Then, in the stationary case 
such regularised solutions are uniquely defined by its values on the vertices of the domain (analogous to a globally continuous solution of the original problem in $\D_n$,  \cf section~10 
in \cite{THJ5}). It just needs to be shown that 
there is sufficient (unique) boundary data.

\subsection*{Acknowledgements}
The research leading to these results has received funding from the European Research Council under the European Union's Seventh Framework Programme (FP7/2007-2013) / ERC grant agreement n$^\circ$~267087. J.\,H. and T.\,D.\,T. have also been supported by scholarships from the IMPRS ``Mathematics in the Sciences'' during earlier stages of this work.

\section{Notation}

\subsection{The simplex}\label{sec_simplex}

As we are considering frequencies (of alleles) in a population, this directly implies the probability simplex as the corresponding state space. In this subsection, we will recall the simplex notation from \cite{THJ4} as well as the appropriate function spaces. 

Let $p^0,p^1,\dots ,p^n$ denote the relative frequencies of alleles $0,1,\dots,n$. As we have $\sum_{j=0}^n p^j=1$ $\Leftrightarrow$ $p^0=1-\sum_{i=1}^n p^i$, this leads to an $n$-dimensional state space
\begin{align}
\D_n =\Bigset{{(p^0,\dotsc,p^n)\in\R^{n+1}\big\vert p^j > 0\text{ for }j=0,1,\dotsc,n \text{ and }\sum_{j=0}^n p^j=1}}
\end{align}
or equivalently 
\begin{align}\label{eq_stand_simpl}
\D_n\ce\Bigset{{(p^1,\dotsc,p^n)\in\R^n\big\vert p^i > 0\text{ for $i=1,\dotsc,n$ and }\sum_{i=1}^n p^i < 1}},
\end{align} which is the  (open) \textit{$n$-dimen\-sional standard orthogonal simplex}\sindex[not]{Dn@$\D_n$}.  

The closure of this simplex is 
\begin{align}\label{eq_simpl_I_n}
\cl{\Delta}_{n}=\bigset{{(p^1,\dotsc,p^n)\in\R^n\big\vert p^i \ge 0 \text{ for }i=1,\dotsc,n \text{ and }\sum_{i=1}^n p^i \le 1}}.
\end{align}
In order to include the time parameter $t\in (-\infty,0]$, we also write
\begin{equation*}
  (\D_n)_{-\infty}:=\D_n\times(-\infty,0). 
\end{equation*}

Considering the boundary of the  simplex  $\bd\D_n=\cl{\D}_n\smin\D_n$, all boundary strata, which are 
(sub-)simplices themselves, are called \textit{faces}, from the $(n-1)$-dimen\-sional \textit{facets} down to the 0-dimensional \textit{vertices}. Each subsimplex of dimension $k\leq n-1$ is isomorphic to the $k$-dimen\-sional standard orthogonal simplex $\D_k$. To denote a particular subsimplex, we introduce index sets $I_k=\{i_0,i_1,\dots ,i_k\}\subset \set{0,\dotsc,n}$ with $i_j\neq i_l$ for $j\neq l$
and put 
\begin{align} 
\Delta_{k}^{(I_{k})}\ce\Bigset{{(p^1,\dotsc,p^n)\in{\overline{\Delta}_n}\big\vert p^i > 0\text{ for $i\in I_k$; }p^i=0\text{ for $i\in I_n\smin I_k$}}}.
\end{align}
The index set $I_n$ may be omitted, thus $\D_n =\Delta_{n}^{(I_{n})}$.

Each of the $\binom{n+1}{k+1}$ subsets $I_k$ of $I_n$  corresponds to a  boundary face $\D_k^{(I_k)}$ ($k\leq n-1$). The \textit{$k$-dimen\-sional part of the  boundary $\bd_k\D_n$ of $\D_n$}\sindex[not]{dkDn@$\bd_k\D_n$} is therefore
\begin{align}\label{eq_bd_k}
\bd_k\D_n^{(I_{n})}\ce \bigcup_{I_k\subset I_n}\Delta_k^{(I_k)}\subset \bd\D_n^{(I_{n})}\quad\text{for $0\leq k\leq n-1$}.
\end{align}
For notational consistency, we also put $\bd_n\D_n =\D_n$. This boundary concept can iteratively be  applied to simplices in the boundary of some $\D_l^{(I_l)}$, $I_l\subset I_n$ for $0\leq k < l\leq n$. We thus have
\begin{align}
\bd_k\D_l^{(I_l)}=\bigcup_{I_k\subset I_l}\Delta_k^{(I_k)}\subset \bd\D_l^{(I_l)}.
\end{align}

Regarding the \WF model, the simplex  $\D_k^{(\set{i_0,\dotsc,i_{k}})}$  corresponds to the state where the $k+1$  the alleles $i_0,\dotsc,i_{k}$ are  present in the population. The boundary $\bd_k\D_n$, \ie the union of all corresponding subsimplices, represents the state with any $k+1$ alleles. Specifically considering the set of alleles  $i_0,\dotsc,i_{k}$ corresponding to $\D_k^{(\set{i_0,\dotsc,i_{k}})}$, the elimination of one of the alleles corresponds to a transition to $\bd_{k-1}\D_k^{(\set{i_0,\dotsc,i_{k}})}$.

We also introduce spaces of square integrable functions for our subsequent integral products on $\D_n$ and its faces (which will mainly be used implicitly, for details \cf \cite{THJ2})\footnote{Here, $\leb_k$ stands for the $k$-dimensional Lebesgue measure, but when integrating over some ${\Delta_k^{(I_k)}}$ with $0\notin I_k$, the measure needs to be replaced with the one induced on ${\Delta_k^{(I_k)}}$ by the Lebesgue measure of the containing $\R^{k+1}$ -- this measure, however, will still be denoted by $\leb_k$\label{pag_leb_k} as it is clear from the domain of integration ${\Delta_k^{(I_k)}}$ with either $0\in I_k$ or $0\notin I_k$ which version is actually used.},
\begin{multline}\label{eq_dfi_L2_union}
L^2\Big(\bigcup_{k=0}^n\bd_k\D_n\Big)
\ce\Big\{f\co\cl{\D}_n\too\R\,\Big\vert\,\text{$f\vert_{\bd_k\D_n}$ is $\leb_k$-measurable and}\\\text{$\int_{\bd_k\D_n} \abs{f(p)}^2\,\leb_k(dp) < \infty$ for all $k=0,\dotsc,n$}\Big\}.
\end{multline}

In order to define an extended solution on $\D_n$ and its faces (indicated by a capitalised $U$), we shall in addition need appropriate spaces of pathwise regular functions.  Such a solution needs to be at least of class $C^2$  in every boundary instance (actually, a solution typically always is of class $C^\infty$, which likewise applies to each boundary instance). Moreover, it should stay regular at boundary transitions that reduce the  dimension by one, \ie for $\D_k^{(I_k)}$ and a boundary face $\D_{k-1}\subset\bd_{k-1}\D_k^{(I_k)}$. Globally, we may require that such a property applies to all possible boundary transitions within $\cl{\D}_n$ and define correspondingly for $l\in \N\cup \set{\infty}$ \sindex[not]{Cplcl@$C_p^l\big(\cl{\D}_n\big)$}
\begin{align}\label{eql_reg_pathwise}
U\in C_p^l\big(\cl{\D}_n\big)\,\ceq\,{U}\vert_{\D_d^{(I_d)}\cup \bd_{d-1}\D_d^{(I_{d})}} \in C^l(\D_d^{(I_d)} \cup \bd_{d-1}\D_d^{(I_d)})\quad
\text{for all $I_d\subset I_n$, $1\leq d \leq n$}
\end{align}
with respect to the spatial variables. Likewise, for ascending chains of (sub-)simplices with a more specific boundary condition, we put for index sets $I_k\subset\dotsc\subset I_n$ and again for $l\in \N\cup\set{\infty}$
\sindex[not]{Cplcup@$C_{p_0}^l\Big(\bigcup_{d=k}^n\D_d^{(I_d)}\Big)$}
\begin{multline}\label{eql_reg_ext}
U\in C_{p_0}^l\Big(\bigcup_{d=k}^n\D_d^{(I_d)}\Big) \,\ceq\, 
\begin{cases}
U\vert_{\D_d^{(I_d)}}\text{ is extendable to }
\bar{U}\in C^l(\D_d^{(I_d)} \cup \bd_{d-1}\D_d^{(I_d)})\text{ with}\\
\bar{U}|_{\bd_{d-1}\D_d^{(I_d)}}=U\ind_{\D_{d-1}^{(I_{d-1})}}\ind_{\set{d>k}}\text{ for all $\max(1,k)\leq d \leq n$}
\end{cases}\hspace*{-16pt}
\end{multline}
with respect to the spatial variables.

\subsection{The cube}

We furthermore introduce some notation for the appearing {cubes and their boundary instances}\index{cube}: In conjunction to the definitions for $\D_n$ in subsection~\ref{sec_simplex}, we define\sindex[not]{Sn@$\cube_n$} for $n\in\N$ an \textit{$n$-dimen\-sional cube}\index{cube}  $\cube_n$ as
\begin{align}\label{eq_def_cube}
\cube_n\ce\bigset{(p^{1},\dotsc,p^{n})\big\vert p^{i}\in(0,1)\text{ for $i=1,\dotsc,n$}}.
\end{align}
Analogous to $\D_n$, if we wish to denote the corresponding coordinate indices explicitly, this may be done by providing the coordinate index set $I_{n}^\prime\ce \set{i_1,\dotsc,i_{n}}\subset\set{1,\dotsc,n}$, $i_j\neq i_l$ for $j\neq l$ as upper index of $\cube_{n}$, thus 
\begin{align}\label{eq_cube_I_n}
\cube_{n}^{(I_n^\prime)}=\bigset{(p^{1},\dotsc,p^{n})\big\vert p^{i}\in(0,1)\text{ for $i\in I_n^\prime$}}.
\end{align}
\sindex[not]{Inp@$I_n^\prime$}%
This is particularly useful for boundary instances of the cube (\cf below) or if for other  purposes a certain ordering $(i_j)_{j=0,\dotsc,n}$ of the coordinate indices is needed. For $\cube_n$ itself and if no ordering is needed, the index set may be omitted (in such a case it may be assumed $I_n^\prime\equiv\set{1,\dotsc,n}$ as in equation~\eqref{eq_def_cube}). 
Please note that a primed index set is always assumed to not contain index 0 (resp.\ $i_0=0$, which we usually stipulate in case of orderings) as the cube does not encompass a \zeroth coordinate. 

In the standard topology on $\R^n$, $\cube_{n}$ is open (which we always assume when writing~$\cube_{n}$), and its closure $\cl{\cube}_n$ is given by 
(again using the index set notation)
\begin{align}
\cl{\cube^{(I_n^\prime)}_n}=\bigset{(p^{1},\dotsc,p^{n})\big\vert p^{i}\in[0,1]\text{ for $i\in I_n^\prime$}}.
\end{align}

Similarly to the simplex, the boundary $\bd\cube_{n}$ of $\cube_{n}$ consists of various subcubes\index{cube!face} (faces) of descending dimensions, starting from the $(n-1)$-dimen\-sional {facets} down to the {vertices}\index{cube!vertex} (which represent 0-dimen\-sional cubes). All appearing subcubes of dimension $0\leq k\leq n-1$ are isomorphic to the $k$-dimen\-sional standard cube $\cube_{k}$ and hence will be denoted by $\cube_{k}$ if it is irrelevant or given by the context which subcube exactly shall be addressed. However, we may state $\cube_{k}^{(I_{k}^\prime)}$ with the index set  $I_{k}^\prime\ce \set{i_1,\dotsc,i_{k}}\subset I_n^\prime$, $i_j\neq i_l$ for $j\neq l$ stipulating that $I_{k}^\prime$ lists all~$k$ `free' coordinate indices, whereas the remaining coordinates are fixed at zero, \ie
\sindex[not]{SkIkp@$\cube_{k}^{(I_k^\prime)}$}
\begin{align} 
\cube_{k}^{(I_{k}^\prime)}\ce\bigset{(p^{1},\dotsc,p^{n})\big\vert p^i \in (0,1)\text{ for $i\in I_k^\prime$; }p^i=0\text{ for $i\in I_n^\prime\smin I_k^\prime$}}
\end{align}
down until $\cube_{0}^{(\varnothing)}\ce(0,\dotsc,0)$ for $k=0$. 

For a given~$k$, there are of course $\binom{n}{k}$ different (unordered) subsets $I_k^\prime$ of $I_n^\prime$, each of which corresponds to a certain boundary face $\cube_{k}^{(I_k^\prime)}$. Moreover, for each subset $I_k^\prime$ with~$k$ elements, altogether $2^{(n-k)}$ subcubes of dimension~$k$ exist in $\bd\cube_n$, which are isomorphic to $\cube_{k}^{(I_{k}^\prime)}$ (including $\cube_{k}^{(I_{k}^\prime)}$) depending on the (respectively fixed) values of the coordinates with indices not in $I_k^\prime$. Thus, if necessary, we may rather state a certain boundary face $\cube_k$ of $\bd\cube_n$ for $0\leq k\leq n-1$ by only giving the values of the $n-k$ fixed coordinates, \ie with indices in $I_n^\prime\smin I_k^\prime$, which may be either 0 or 1, hence
\begin{align}
\cube_{k}=\bigset{p^{j_1}=b_{1},\dotsc,p^{j_{n-k}}=b_{n-k}}
\end{align}
with $j_1,\dotsc,j_{n-k}\in I^\prime_n$, $i_r\neq i_s$ for $r\neq s$ and $b_{1},\dotsc,b_{n-k}\in\set{0,1}$ chosen accordingly. 
In particular for dimension $n-1$, it is noted that we have $n-1$ faces, which each appear twice; in zero dimension, there are $2^n$ vertices. If we wish to indicate the total \textit{$k$-dimen\-sional boundary of $\cube_n$}\index{cube!k-dimen\-sional boundary@$k$-dimen\-sional boundary}, \ie the union of all $k$-dimen\-sional faces belonging to $\cl{\cube}_n$, we may write $\bd_k\cube_n$ for $k=0,\dotsc,n$ with analogously $\bd_n\cube_n\ce\cube_n$.
\sindex[not]{dkS@$\bd_k\cube_n$}

Lastly, when writing products of simplex and cube which do not span all considered dimensions, we indicate the value of the missing coordinates by curly brackets marked with the corresponding coordinate index, \ie for $I_n=\set{i_0,i_1,\dotsc,i_n}$ and $I_k\subset I_n$ with $i_{k+1}\notin I_k$ we have \eg
\begin{align}\notag
&\D_k^{(I_k)}\times\set{1}^{(\set{i_{k+1}})}\times\cube_{n-k-1}^{(I_n^\prime\smin(I_k^\prime\cup\set{i_{k+1}}))}\\
&\hspace*{9pt}\ce\bigset{(p^{i_1},\dotsc,p^{i_n})\big\vert p^i>0\text{ for $i\in I_k$},p^{i_{k+1}}=1,p^j\in(0,1)\text{ for $j\in I_n^\prime\smin (I_k^\prime\cup\set{i_{k+1}})$}}
\end{align}
with $p^{i_0}=p^0=1-\sum_{j=1}^{k}p^{i_j}$. If coordinates are fixed at 0, the corresponding entry may be omitted, \eg we may just write $\D_k^{(I_k)}$ for $\D_k^{(I_k)}\times\set{0}^{(I_n\smin I_k)}$.

Furthermore, we also introduce a (closed) cube $\cl{\cube_k^{(I^\prime_k)}}$ with a removed base vertex $\cube_0^{(\varnothing)}$ somewhat inexactly denoted by $\cl{\cubex_k^{(I^\prime_k)}}$, \sindex[not]{Sx@$\cl{\cubex_k^{(I^\prime_k)}}$} \ie 
\begin{align}
\cl{\cubex_k^{(I^\prime_k)}}\ce\cl{\cube_k^{(I^\prime_k)}}\smin\cube_0^{(\varnothing)}=\Bigset{p^{i_1},\dotsc,p^{i_k}\in [0,1]\Big\vert\sum_{j=1}^k p^{i_j}>0}.
\end{align}

For functions defined on the cube, the pathwise smoothness required  for an application of the yet to be introduced corresponding Kolmogorov backward operator\index{backward operator} (\cf p.~\pageref{lem_blowup_trans_op}) may be defined as with the simplex in equality~\eqref{eql_reg_pathwise} in \cite{THJ5}; 
hence, we put \sindex[not]{Cplcl@$C_p^l(\cl{\cube}_n)$}
\begin{align}
\tilde{u}\in C_p^l(\cl{\cube}_n)\,\ceq\,{\tilde{u}}\vert_{\cube_d\cup \bd_{d-1}\cube_d} \in C^l(\cube_d \cup \bd_{d-1}\cube_d)\text{ for every $\cube_d\subset\cl{\cube}_n$}
\end{align}
with respect to the spatial variables, implying that the operator is continuous at all boundary transitions within $\cl{\cube}_n$. This concept likewise applies to subsets of $\cl{\cube}_n$ where needed.
%
%

\section{Hierarchical extended solutions of the \KBE}

In this section, we will provide the principle results from~\cite{THJ5}; for details please see also there. The class of extensions in consideration is confined by the extension constraints (definition 6.1 in \cite{THJ4}):

\begin{dfi}[extension constraints]\label{dfi_ext}
Let $I_d$ be an index set with $\abs{I_d}=d+1\geq 2$, $0,s\in I_d$ and $\D_d^{(I_d)}=\set{(p^i)_{i\in I_d\setminus\set{0}}\vert p^i>0\text{ for $i\in I_d$}}$  with $p^0\ce 1-\sum_{i\in I_d\setminus\set{0}}p^i$. For $d\geq2$ and a solution $u\co \big({\D_{d-1}^{(I_d\setminus\set{s})}}\big)_{-\infty}\too \R$ of the correspondingly restricted \KBE~\eqref{eq_back_n_ext}, 
\ie ${u}(\fdt,t)\in C^\infty\big(\D_{d-1}^{(I_d\smin\set{s})}\big)$ for $t<0$,  $u(p,\,\cdot\,)\in C^\infty((-\infty,0))$ for $p\in\D_{d-1}^{(I_d\smin\set{s})}$ and
\begin{align}
-\dd{}{t}u= L^*u
\quad\text{in $\big({\D_{d-1}^{(I_d\smin\set{s})}}\big)_{-\infty}$},
\end{align}
a function $\bar{u}\co \big({\D_{d}^{(I_d)}}\big)_{-\infty}\too \R$ with $\bar{u}(\fdt,t)\in C^\infty\big(\D_d^{(I_d)}\big)$ for $t<0$ and $\bar{u}(p,\,\cdot\,)\in C^\infty((-\infty,0))$ for $p\in\D_{d}^{(I_d)}$
is said to be an extension of $u$ in accordance with the \textit{extension constraints} if
\begin{itemize}
 \item[(i)]{for~$t<0$ $\bar{u}(\fdt,t)$ is continuously extendable to the boundary $\bd_{d-1}\D_d^{(I_d)}$ such that it coincides with $u(\fdt,t)$ in $\D_{d-1}^{(I_d\setminus\set{s})}$ \resp vanishes on the remainder of $\bd_{d-1}\D_d^{(I_d)}$ and is of class $C^\infty$ with respect to the spatial variables in $\D_d^{(I_d)}\cup\bd_{d-1}\D_d^{(I_d)}$,}
 \item[(ii)]{it is a solution of the corresponding \KBE in $\big({\D_d^{(I_d)}}\big)_{\hspace*{-2pt}-\infty}$\hspace*{-2pt}, \ie $-\dd{}{t}\bar{u}= L^*\bar{u}$ in $\big({\D_d^{(I_d)}}\big)_{-\infty}$.}
\end{itemize}
For $d=1$, this analogously applies to functions $u$ with $-\dd{}{t}u=0$ (in accordance with  $L^*_0\equiv 0$), and consequently the equation in condition~(ii) is replaced with $L^*\bar{u}=0$. 
Furthermore, an extension which encompasses multiple extension steps is said to be in accordance with the {extension constraints}, if this holds for every extension step.
\end{dfi}

The presented extension scheme then first yields the extensions of single extensions of solutions from a boundary instance of the considered domain to the interior (proposition 6.4 in \cite{THJ5}), from which one can advance to the existence of pathwise extensions (proposition 8.1 in \cite{THJ5}):

\begin{prop}[pathwise extension of solutions]\label{prop_ext_iter}
Let $k,n\in\N$ with $0\leq k<n$, $\set{i_k,i_{k+1},\dotsc,i_n}\subset I_n\ce\set{0,1,\dotsc,n}$ with $i_i\neq i_j$ for $i\neq j$ and $I_k\ce I_n\setminus\set{i_{k+1},\dotsc,i_n}$, and let $u_{I_k}$ \sindex[not]{uik@$u_{I_k}$} be a proper solution of the \KBE \eqref{eq_back_n_ext} restricted to ${\D_k^{(I_k)}}$ for some final condition $ f\in \L^2\big({\D_k^{(I_k)}}\big)$. 
For $d=k+1,\dotsc,n$ and $I_{d}\ce I_k \cup \set{i_{k+1},\dotsc i_{d}}$, an extension of $\bar{u}_{I_k}^{i_k,\dotsc,i_{d-1}}$ in $\big(\D_{d-1}^{(I_{d-1})}\big)_{-\infty}$ to  $\bar{u}_{I_k}^{i_k,\dotsc,i_{d}}\ce{\big(\bar{u}_{I_k}^{i_k,\dotsc,i_{d-1}}\big)}^{i_{d-1},i_d}$ in $\big(\D_d^{(I_{d})}\big)_{-\infty}$ as by proposition 6.4 in \cite{THJ5} is in accordance with the extension constraints~\ref{dfi_ext} if (and for $d\geq k+2$ and $[f]\neq0$ in $L^2\big({\D_k^{(I_k)}}\big)$ also only if) putting $r(d)=i_{d-1}$ for the extension target face index, and we respectively have 
%
\sindex[not]{uikid@$\bar{u}_{I_k}^{i_k,\dotsc,i_{d}}$}\sindex[not]{pikid@$\pi^{i_k,\dotsc,i_d}$}
\begin{align}\label{eq_iter_ext}
\bar{u}_{I_k}^{i_k,\dotsc,i_{d}}(p,t)=
u_{I_k}(\pi^{i_k,\dotsc,i_d}(p),t)\prod_{j=k}^{d-1} \frac{p^{i_j}}{\sum_{l=j}^d p^{i_l}},\quad (p,t)\in\big(\D_d^{(I_d)}\big)_{-\infty}
\end{align}
with $p^0=1-\sum_{i\in I_d\setminus\set{0}}p^i$ and $\pi^{i_k,\dotsc,i_d}(p)=(\tilde{p}^1,\dotsc,\tilde{p}^n)$
such that $\tilde{p}^{i_k}=p^{i_k}+\dotso+p^{i_d}$, $\tilde{p}^{i_{k+1}}=\dotso=\tilde{p}^{i_{d}}=0$ and $\tilde{p}^j=p^j$ for $j\in I_d\smin\set{i_k,\dotsc,i_d}$.


Correspondingly, the resulting assembling of all extensions to a function  $\bar{U}_{I_k}^{i_k,\dotsc,i_{n}}$ in $\Big(\bigcup_{ k \leq d\leq n}\D_d^{(I_d)}\Big)_{-\infty}$
by putting\sindex[not]{Uikid@${U}_{I_k}^{i_k,\dotsc,i_{n}}$}
\begin{multline}\label{eq_path_ext}
\bar{U}_{I_k}^{i_k,\dotsc,i_{n}}(p,t) \ce u_{I_k}(p,t)\ind_{\D_k^{(I_k)}}(p)+\sum_{ k+1\leq d\leq n}\bar{u}_{I_k}^{i_k,\dotsc,i_{d}}(p,t)\ind_{\D_d^{(I_d)}}(p)\\
=u_{I_k}(p,t)\ind_{\D_k^{(I_k)}}(p)+\sum_{ k+1\leq d\leq n} u_{I_k}(\pi^{i_k,\dotsc,i_d}(p),t)\prod_{j=k}^{d-1} \frac{p^{i_j}}{\sum_{l=j}^d p^{i_l}} \ind_{\D_d^{(I_d)}}(p)
\end{multline}
with $p^0=1-\sum_{i\in I_n\setminus\set{0}}p^i$ is in $C_{p_0}^\infty\Big(\bigcup_{ k \leq d\leq n}\D_d^{(I_d)}\Big)$ with respect to the spatial variables for $t<0$ 
as well as in $C^\infty((-\infty,0))$ with respect to $t$, and we have
\begin{align}
\begin{cases}\label{eq_iter_ext_sol}
L^*\bar{U}_{I_k}^{i_k,\dotsc,i_{n}}=-\dd{}{t}\bar{U}_{I_k}^{i_k,\dotsc,i_{n}}
&\text{in $\Big(\bigcup_{ k \leq d\leq n}\D_d^{(I_d)}\Big)_{-\infty}$}\\
\bar{U}_{I_k}^{i_k,\dotsc,i_{n}}(\fdt,0)=\bar{F}_{I_k}^{i_k,\dotsc,i_{n}}
&\text{in $\bigcup_{ k \leq d\leq n}\D_d^{(I_d)}$}
\end{cases}
\end{align}
with $\bar{F}_{I_k}^{i_k,\dotsc,i_{n}}\in \L^2\Big(\bigcup_{ k \leq d\leq n}\D_d^{(I_d)}\Big)$ being an analogous extension of the final condition $f=f_{I_k}$ in $\D_k^{(I_k)}$; in particular, we have $\bar{U}_{I_k}^{i_k,\dotsc,i_{n}}\big|_{{\D_k^{(I_k)}}}(\fdt,0)=f$ in~${{\D_k^{(I_k)}}}$.
\end{prop}

From this, a corresponding existence result by Littler (\cf \cite{littler-good_ages}) may be reconstructed. Assembling all pathwise extensions eventually yields the existence of the global extensions (proposition 8.4 in \cite{THJ5}). At last, a solution scheme for the extended \KBE based on the global iterative extensions is presented, yielding the following existence result (theorem 9.1 in \cite{THJ5}):
\begin{thm}\label{thm_sol_back_n_ext}
For a given final condition $f\in\L^2\big(\bigcup_{d=0}^n\bd_d\D_n\big)$,  the extended \KBE\eqref{eq_back_n_ext} corresponding to the $n$-dimen\-sional \WF model  in diffusion approximation  always allows a solution 
$\bar{U}\co{\big(\overline{\Delta}_{n}\big)}_{-\infty}\map\R$ with $\bar{U}(\,\cdot\,,t)\in C_p^\infty\big(\cl{\D}_n\big)$ for each fixed $t\in(-\infty,0)$  
and $\bar{U}(p,\,\cdot\,)\in C^\infty((-\infty,0))$ for each fixed $p\in\overline{\Delta}_{n}$.
\end{thm}

\section{Motivation}

To illustrate the motivation for the regularisation scheme, we use the example of $\bar{U}_{I_k}^{i_k,\dotsc,i_{n}}$ in $\cl{\D_n^{(I_{n})}}$ as in  equation~\eqref{eq_path_ext}: An assessment of the geometrical situation of the respective incompatibilities shows that for every $t<0$ the critical area for the top-dimen\-sional component $\bar{u}_{I_k}^{i_k,\dotsc,i_{n}}$ \resp its continuous extension actually only consists of the domain where we have $p^{i_n}+p^{i_{n-1}}=0$, hence $\cl{\D_{n-2}^{(I_{n-2})}}$. On all other boundary instances of arbitrary dimension, $\bar{u}_{I_k}^{i_k,\dotsc,i_{n}}$ as in equation~\eqref{eq_iter_ext} is continuously extendable and of class $C^\infty$ with respect to the spatial variables there. Thus, at first there is only one connected component of the boundary gap which needs to be addressed. 

However, as will turn out, the full hierarchical solution $\bar{U}_{I_k}^{i_k,\dotsc,i_{n}}$ actually comprises a nested incompatibility in $\cl{\D_{n-2}^{(I_{n-2})}}$ in the sense that also $\bar{u}_{I_k}^{i_k,\dotsc,i_{n-1}}$ does not extend continuously to $\cl{\D_{n-3}^{(I_{n-3})}}$ and so forth until $\bar{u}_{I_k}^{i_k,i_{k+1},i_{k+2}}$ not extending continuously to $\cl{\D_{k}^{(I_{k})}}$. This implies that the desired transformation needs to affect all relevant dimensions, which will be accomplished by an iterative advancement: In each step, one dimension from the simplex is removed and converted into a dimension of the corresponding cube component, \ie the corresponding coordinate is released from the simplex property $\sum_i p^i\leq1$. In doing so, the solution gains the required regularity at the corresponding level with each iteration, \ie eventually each of its components is transformed such that it extends smoothly to the boundary. Thus, after $n-k-1$ of these steps, the relevant component of  $\cl{\D_n^{(I_n)}}$ is converted into a cube of dimension $n-k-1$, and the correspondingly transformed solution 
is sufficiently regularised, in particular meaning that it now smoothly extends to the full boundary, as will be shown.

\section{The blow-up transformation and its iteration}

Now, we will present the \txind{blow-up transformation} in full detail and state all necessary results. We start with the findings for the basic transformation and  advance to the results for a suitably iterated application of the blow-up transformation later:  

\begin{lem}[blow-up transformation]\label{lem_blowup_trans}
Let $I_d=\set{0,1,\dotsc,d}$. A blow-up transformation $\Phi^\si_\ri$ with $\si,\ri\in I_d\smin\set{0}$ mapping
\begin{align}
\cl{\D_d^{(I_d)}}\smin\cl{\D_{d-2}^{(I_d\smin\set{\si,\ri})}}
=\bigset{(p^1,\dotsc,p^d)\big\vert p^i\geq0\text{ for $i\in I_d$},p^\si+p^\ri>0}
\end{align}
with ${p}^0\ce1-\sum_{i\in I_d\smin\set{0}} {p}^i$ $C^\infty$-diffeomorphically onto
\begin{multline}
\Big(\cl{\D_{d-1}^{(I_d\smin\set{\ri})}}\smin\cl{\D_{d-2}^{(I_d\smin\set{\si,\ri})}}\Big)\times\cl{\cube_{1}^{(\set{\ri})}}\\
=\bigset{(\tilde{p}^1,\dotsc,\tilde{p}^d)\big\vert\tilde{p}^i\geq0\text{ for $i\in I_d\smin\set{\ri}$},\tilde{p}^\si>0;\tilde{p}^\ri\in[0,1]}
\end{multline}
with $\tilde{p}^0\ce1-\sum_{i\in I_d\smin\set{0,\ri}} \tilde{p}^i$
and altogether
\begin{align}
 \cl{\D_d^{(I_d)}} \longmapsto \Big( \cl{\D_{d-1}^{(I_d\smin\set{\ri})}}\times\cl{\cube_{1}^{(\set{\ri})}}\Big)\smin N_{\si}
\end{align}
with\sindex[not]{Nr@$N_\si$}
\begin{align}
N_{\si}\ce \cl{\D_{d-2}^{(I_d\smin\set{\si,\ri})}}\times\set{0}^{(\set{\si})}\times\cl{\cubex_1^{(\set{\ri})}},
\end{align}
appearing as an additional $(d-1)$-dimen\-sional face \index{cube!additional face} of $\cl{\D_{d-1}^{(I_d\smin\set{\ri})}}\times\cl{\cube_{1}^{(\set{\ri})}}$,
is given by
\begin{align}\label{eq_blowup_trans_i}
\tilde{p}^i&\ce\hspace*{0.9em} p^i\hspace*{2.4em}\text{for $i\neq \si,\ri$},\\
\label{eq_blowup_trans_si}
\tilde{p}^\si&\ce\hspace*{0.9em} p^\si + p^\ri,\\\label{eq_blowup_trans_ri}
\tilde{p}^\ri&\ce
\begin{cases}
 \frac{p^\ri}{p^\si + p^\ri}&\text{for $p^\si+p^\ri>0$}\\
 0 	 &\text{for $p^\si+p^\ri=0$}.
\end{cases}
\end{align}

\end{lem}

\begin{cor}\label{cor_blowup_trans}
While we obtain $N_r=\cl{\D_{d-2}^{(I_d\smin\set{\si,\ri})}}\times\cl{\cubex_{1}^{(\set{\ri})}}$ as an additional $(d-1)$-dimen\-sional face with $\Phi^\si_\ri$, the existing $(d-1)$-dimen\-sional faces of $\cl{\D_d^{(I_d)}}$ including their boundaries are mapped as follows:
\begin{align}
\cl{\D_{d-1}^{(I_d\smin\set{\ri})}}&\longmapsto \cl{\D_{d-1}^{(I_d\smin\set{\ri})}}\times\set{0}^{(\set{\ri})},\\
\cl{\D_{d-1}^{(I_d\smin\set{\si})}}\smin\cl{\D_{d-2}^{(I_d\smin\set{\si,\ri})}}
&\longmapsto\Big(\cl{\D_{d-1}^{(I_d\smin\set{\ri})}}\smin\cl{\D_{d-2}^{(I_d\smin\set{\si,\ri})}}\Big)\times\set{1}^{(\set{\ri})}
\end{align}
and
\begin{align}
\cl{\D_{d-1}^{(I_d\smin\set{i})}}\smin\cl{\D_{d-3}^{(I_d\smin\set{i,\si,\ri})}}
&\longmapsto\Big(\cl{\D_{d-2}^{(I_d\smin\set{i,\ri})}}\smin\cl{\D_{d-3}^{(I_d\smin\set{i,\si,\ri})}}\Big)\times\cl{\cube_{1}^{(\set{\ri})}}
\quad\text{for $i\in I_d\smin\set{\si,\ri}$}.
\end{align}
\end{cor}

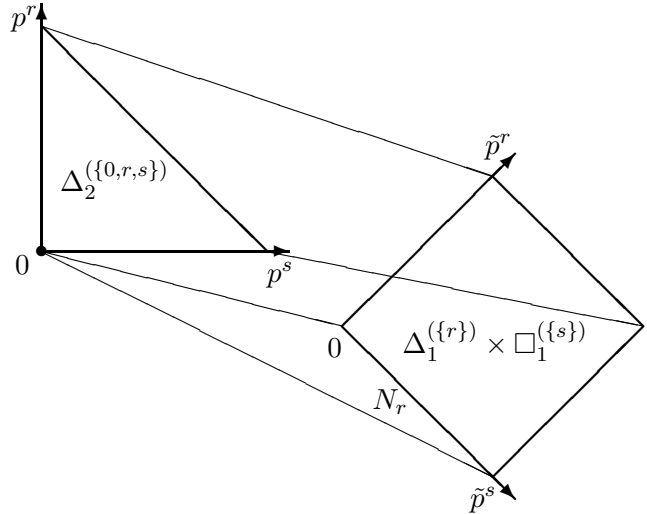
\begin{figure}[ht]\centering
\setlength{\unitlength}{1cm}
\begin{picture}(8,7)(-0.3,-0.5)

\thicklines
\put(0,3){
\put(0,0){\vector(0,1){3.3}}
\put(0,0){\vector(1,0){3.3}}
\put(0,0){\circle*{0.15}}
\put(0,3){\line(1,-1){3}}

\put(-0.4,3){$p^\si$}
\put(3,-0.4){$p^\ri$}
\put(-0.35,-0.3){$0$}
\put(0.25,0.8){$\D_2^{(\set{0,\si,\ri})}$}
}

\thicklines
\put(4,2){
\put(0,0){\vector(1,1){2.3}}
\put(0,0){\vector(1,-1){2.3}}
\put(2,2){\line(1,-1){2}}
\put(2,-2){\line(1,1){2}}

\put(1.9,2.3){$\tilde{p}^\si$}
\put(1.7,-2.4){$\tilde{p}^\ri$}
\put(-0.2,-0.4){$0$}
\put(0.8,-0.35){$\D_1^{(\set{\si})}\times\cube_1^{(\set{\ri})}$}
\put(0.4,-1.1){$N_\si$}
}

\thinlines
\put(0,6){\line(3,-1){6}}
\put(3,3){\line(5,-1){5}}

\thinlines
\put(0,3){\line(4,-1){4}}
\put(0,3){\line(2,-1){6}}
\end{picture}
\caption{An illustration of the blow-up transformation for $d=2$}
\end{figure}

\begin{rmk}\label{rmk_blowup_trans}
If the $\tilde{p}^\ri$ in lemma~\ref{lem_blowup_trans} is chosen differently with
\begin{align}
{\tilde{p}^\ri}&\ce\frac{p^\si}{p^\si + p^\ri},
\end{align}
this flips the orientation of the $\tilde{p}^\ri$-coordinate in $\cube_{1}^{(\set{\ri})}$ as all appearances of $\tilde{p}^\ri$ now need to be replaced by $1-\tilde{p}^{\ri}$. This, however, does not affect the statements of lemma~\ref{lem_blowup_trans}, whereas in corollary~\ref{cor_blowup_trans} the images of $\cl{\D_{d-1}^{(I_d\smin\set{\si})}}\smin\cl{\D_{d-2}^{(I_d\smin\set{\si,\ri})}}$ and $\cl{\D_{d-1}^{(I_d\smin\set{\ri})}}\smin\cl{\D_{d-2}^{(I_d\smin\set{\si,\ri})}}$ are interchanged. Thus, unless stated differently, in the following we will always assume that the $\tilde{p}^\ri$-coordinate is chosen with an orientation as given in lemma~\ref{lem_blowup_trans}.
\end{rmk}

\begin{proof}[Proof of lemma~\ref{lem_blowup_trans}]
The transformation corresponds geometrically to 
a scaling of the domain into $\tilde{p}^\ri$-direction with the scaling factor given by $\frac1{\tilde{p}^{\si}}$. 
The assertion about the transformation domains is straightforward since we have $0\leq\frac{p^\ri}{p^\si + p^\ri}\leq1$ on $\cl{\D_d^{(I_d)}}\smin\cl{\D_{d-2}^{(I_d\smin\set{\si,\ri})}}$. Likewise, the $C^\infty$-diffeomorphism property follows from $\Phi^\si_\ri$ being smoothly differentiable as long as $\tilde{p}^\si= p^\si + p^\ri>0$ and the smoothness of the inverse transformation $(\Phi^\si_\ri)^{-1}$,  given by
\begin{align}\label{eq_blowup_inverse}
{p^\si}&=\tilde{p}^\si(1-\tilde{p}^\ri),\\
{p^\ri}&=\tilde{p}^\si\tilde{p}^\ri,\\
{p}^i&=\tilde{p}^i\quad\text{for $i\neq \si,\ri$}.
\end{align}
By this, it also becomes obvious that $(\Phi^\si_\ri)^{-1}$ maps $\Big(\cl{\D_{d-1}^{(I_d\smin\set{\ri})}}\smin\cl{\D_{d-2}^{(I_d\smin\set{\si,\ri})}}\Big)\times\cl{\cube_{1}^{(\set{\ri})}}$ onto $\cl{\D_d^{(I_d)}}\smin\cl{\D_{d-2}^{(I_d\smin\set{\si,\ri})}}$. 
\end{proof}

The next lemma is concerned with the transformation behaviour of the operator $L^*_n$; all considerations apply to $L^*_n$ in its domain $\D_n$ as well as  -- considering the restrictability of $L_n^*$ (\cf 
\cite{THJ5}) -- in the closure $\cl{\D}_n$ \resp to the transformed operator $\tilde{L}_n^*$ in the subsequent transformation images of the domain (the domain in question may not be stated explicitly -- this will be done in proposition~\ref{prop_blowup_trans_total}):

\begin{lem}\label{lem_blowup_trans_op}
Let $I_n^\prime\ce\set{1,\dotsc,n}$ be an index set with $\si,\ri\in I_n^\prime$ and let $\set{i_{1},\dotsc,i_{n}}$ be an ordering of $I_n^\prime$ such that $\si,\ri\in\set{i_{1},\dotsc,i_{m}}$ for some $m\leq n$. When changing coordinates $(p^i)_{i\in I_n^\prime}\mapsto(\tilde{p}^i)_{i\in I_n^\prime}$
by $\Phi^\si_\ri$, the operator
\begin{align}
L_n^*=\half\sum_{i,j=1}^n a^{ij}(p)\dd{}{p^i}\dd{}{p^j}
\end{align}
with $a^{ij}(p)=p^i(\delta^i_j-p^j)$ for $i,j\in\set{i_{1},\dotsc,i_{m}}$, $a^{ij}=0$ else for $i\neq j$ is transformed into
\begin{align}
\tilde{L}_n^*=\half\sum_{k,l=1}^k\tilde{a}^{kl}(\tilde{p})\dd{}{\tilde{p}^k}\dd{}{\tilde{p}^l}
\end{align}
with $\tilde{a}^{kl}(\tilde{p})=\tilde{p}^k(\delta^k_l-\tilde{p}^l)$ for $k,l\in\set{i_{1},\dotsc,i_{m}}\smin\set{\ri}$, $\tilde{a}^{\ri\ri}(\tilde{p})=\frac{\tilde{p}^{\ri}(1-\tilde{p}^\ri)}{\tilde{p}^\si}$, $\tilde{a}^{\ri l}=\tilde{a}^{l\ri}=0$ for $l\neq \ri$ and $\tilde{a}^{kl}(\tilde{p})=a^{kl}(p)$ (with the coordinates yet to be replaced) for all remaining indices. This also holds if the $\tilde{p}^\ri$-coordinate is chosen with opposite orientation (\cf remark~\ref{rmk_blowup_trans}).
\end{lem}

\begin{proof}
Under a change of coordinates $(p^i)\mapsto(\tilde{p}^i)$, the coefficients of the \second order derivatives $a^{ij}$ transform as
\begin{align}
\tilde{a}^{kl}=\sum_{i,j}a^{ij}\dd{\tilde{p}^k}{p^i}\dd{\tilde{p}^l}{p^j},
\end{align}
while we may get additional first order derivatives with coefficients $\sum_{i,j}a^{ij}\ddd{\tilde{p}^k}{p^i}{p^j}$.

For the transformation at hand, we have (\cf equations~\eqref{eq_blowup_trans_si} and~\eqref{eq_blowup_trans_i})
\begin{align}
\dd{\tilde{p}^k}{p^i}=\delta^k_i+\delta^k_\si\delta^\ri_i\quad\text{for $k\neq \ri$}
\end{align}
and (\cf equation~\eqref{eq_blowup_trans_ri})
\begin{align}\label{eq_blowup_trans_t_derived}
\dd{\tilde{p}^\ri}{p^i}
=\frac{p^\si}{(p^\si+p^\ri)^2}\delta^\ri_i-\frac{p^\ri}{(p^\si+p^\ri)^2}\delta^\si_i
=\frac{1-\tilde{p}^\ri}{\tilde{p}^\si}\delta^\ri_i-\frac{\tilde{p}^\ri}{\tilde{p}^\si}\delta^\si_i.
\end{align}
Utilising this, we obtain
\begin{align}
\tilde{a}^{kl}(\tilde{p})
&=\sum_{i,j}a^{ij}(p)(\delta^k_i+\delta^k_\si\delta^\ri_i)(\delta^l_j+\delta^l_\si\delta^\ri_j)
\end{align}
for $k,l\neq \ri$, yielding
\begin{align}\notag
\tilde{a}^{kl}(\tilde{p})
&=a^{kl}(p)+a^{kt}(p)\delta^l_\si+a^{\ri l}(p)\delta^k_\si+a^{\ri\ri}(p)\delta^k_\si\delta^l_\si\\\notag
&=p^k(\delta^k_l-p^l)-p^k p^\ri\delta^l_\si-p^\ri p^l\delta^k_\si+p^\ri(1-p^\ri)\delta^k_\si\delta^l_\si\\
&=\tilde{p}^k(\delta^k_l-\tilde{p}^l)
\end{align}
for $k,l\in\set{i_{1},\dotsc,i_{m}}\smin\set{\ri}$ using the given form of the $a^{ij}$, whereas for all other index pairs not containing the index~$t$, we always have
\begin{align}
a^{kt}(p)\delta^l_\si=a^{\ri l}(p)\delta^k_\si=a^{\ri\ri}(p)\delta^k_\si\delta^l_\si=0
\end{align}
and hence
\begin{align}
\tilde{a}^{kl}(\tilde{p})=\sum_{i,j}a^{ij}(p)\delta^k_i\delta^l_j=a^{kl}(p),
\end{align}
thus proving the last statement. Furthermore,  we have for arbitrary $l\neq \ri$
\begin{align}\notag
\tilde{a}^{\ri l}(\tilde{p})&=\sum_{i,j}a^{ij}(p)\bigg(\frac{1-\tilde{p}^\ri}{\tilde{p}^\si}\delta^\ri_i-\frac{\tilde{p}^\ri}{\tilde{p}^\si}\delta^\si_i\bigg)(\delta^l_j+\delta^l_\si\delta^\ri_j)\\\notag
&=\frac{1-\tilde{p}^\ri}{\tilde{p}^\si}(a^{\ri l}(p)+a^{\ri\ri}(p)\delta^l_\si )-\frac{\tilde{p}^\ri}{\tilde{p}^\si}(a^{\si l}(p)+a^{\si t}(p)\delta^l_\si )\\\notag
&=\Big(-\frac{1-\tilde{p}^\ri}{\tilde{p}^\si}\tilde{p}^\si\tilde{p}^\ri \tilde{p}^l+\frac{\tilde{p}^\ri}{\tilde{p}^\si}(1-\tilde{p}^\ri)\tilde{p}^\si \tilde{p}^l\Big)\ind_{\set{i_{1},\dotsc,i_{m}}}(l)\\ 
&\quad-\frac{\tilde{p}^\ri}{\tilde{p}^\si}\tilde{p}^\si(1-\tilde{p}^\ri)\delta^l_\si+ \Big(\frac{1-\tilde{p}^\ri}{\tilde{p}^\si}\tilde{p}^\si\tilde{p}^\ri(1-\tilde{p}^\si\tilde{p}^\ri)
+\frac{\tilde{p}^\ri}{\tilde{p}^\si}\tilde{p}^\si(1-\tilde{p}^\ri)\tilde{p}^\si\tilde{p}^\ri\Big)\delta^l_\si=0%
\end{align}
as well as $\tilde{a}^{l\ri}=0$ ($l\neq \ri$) by symmetry and eventually
\begin{align}\notag
\tilde{a}^{\ri\ri}(\tilde{p})
&=\sum_{i,j}a^{ij}(p)\bigg(\frac{1-\tilde{p}^\ri}{\tilde{p}^\si}\delta^\ri_i-\frac{\tilde{p}^\ri}{\tilde{p}^\si}\delta^\si_i\bigg)\bigg(\frac{1-\tilde{p}^\ri}{\tilde{p}^\si}\delta^\ri_j-\frac{\tilde{p}^\ri}{\tilde{p}^\si}\delta^\si_j\bigg)\\\notag
&=a^{\ri\ri}(p)\bigg(\frac{1-\tilde{p}^\ri}{\tilde{p}^\si}\bigg)^2
+a^{\si \si}(p)\bigg(\frac{\tilde{p}^\ri}{\tilde{p}^\si}\bigg)^2
-2a^{\ri\si}(p)\frac{\tilde{p}^\ri(1-\tilde{p}^\ri)}{{(\tilde{p}^\si )}^2}\\\notag
&=\tilde{p}^\ri(1-\tilde{p}^\si\tilde{p}^\ri)\frac{(1-\tilde{p}^\ri)^2}{\tilde{p}^\si}
+(1-\tilde{p}^\ri)(1-\tilde{p}^\si+\tilde{p}^\si\tilde{p}^\ri)\frac{(\tilde{p}^\ri)^2}{\tilde{p}^\si}\\&\quad%
-2\tilde{p}^\si\tilde{p}^\ri(1-\tilde{p}^\ri)\frac{\tilde{p}^\ri(1-\tilde{p}^\ri)}{\tilde{p}^\si}
=\frac{\tilde{p}^\ri(1-\tilde{p}^\ri)}{\tilde{p}^\si},
\end{align}
by which the form of all $\tilde{a}^{kl}$ is shown.

When checking for possible additional first order derivatives, it is obvious that the second order coordinate derivatives do not vanish at first glance only for $\tilde{p}^\ri$. But we have (\cf equation~\eqref{eq_blowup_trans_t_derived})
\begin{align}
\dd{}{p^j}\dd{}{p^i}\tilde{p}^\ri
&=\frac2{(p^\si+p^\ri)^3}(p^\ri\delta^\si_i-p^\si\delta^\ri_i)(\delta^\si_j+\delta^\ri_j)
+\frac1{(p^\si+p^\ri)^2}(\delta^\ri_i\delta^\si_j-\delta^\si_i\delta^\ri_j)
\end{align}
and subsequently
\begin{align}\notag
\sum_{i,j}a^{ij}\dd{}{p^i}\dd{}{p^j}\tilde{p}^\ri
&=\frac{2}{(p^\si+p^\ri)^3}\big(p^\ri(a^{\si \si}+ a^{\si \ri})-p^\si (a^{\ri\si}+a^{\ri\ri})\big)+\frac{1}{(p^\si+p^\ri)^2}(a^{\ri\si}-a^{\si \ri})\\
&=\frac{2}{(p^\si+p^\ri)^3}\big(p^\ri p^\si(1-p^\si-p^\ri)+p^\si p^\ri(p^\si-1+p^\ri)\big)=0,
\end{align}
for which again the specified form of the appearing $a^{ij}$ is needed.

If $\tilde{p}^\ri$ is chosen with different orientation as in remark~\ref{rmk_blowup_trans}, instead of equation~\eqref{eq_blowup_trans_t_derived} we then have
\begin{align}
\dd{\tilde{p}^\ri}{p^i}
=\frac{\tilde{p}^\ri}{\tilde{p}^\si}\delta^\ri_i-\frac{1-\tilde{p}^\ri}{\tilde{p}^\si}\delta^\si_i,
\end{align}
signifying that in the respective formulae the indices $\si$ and $\ri$ are swapped, which in turn is matched by the corresponding inverse transformation now yielding $p^\si=\tilde{p}^\si\tilde{p}^\ri$ and $p^\ri=\tilde{p}^\si(1-\tilde{p}^\ri)$.
\end{proof}

Combining the preceding results, we obtain for an iterated application of the blow-up transformation\index{blow-up transformation!iterated}:

\begin{prop}\label{prop_blowup_trans_total}
Let $k,n\in\N$ with $0\leq k\leq n-2$, $\set{i_k,i_{k+1},\dotsc,i_n}\subset I_n\ce\set{0,1,\dotsc,n}$ with $i_i\neq i_j$ for $i\neq j$ and $I_d\ce I_n \smin\set{i_{d+1},\dotsc,i_n}$ for $d=k,\dotsc,n-1$.
A repeated blow-up transformation $\Phi^{\si_{n-k-1}}_{\ri_{n-k-1}}\circ\dotso\circ\Phi^{\si_1}_{\ri_{1}}$ with $\Phi^{\si_\m}_{\ri_\m}$ as in lemma~\ref{lem_blowup_trans} with $\si_\m=i_{n-\m}$ and $\ri_\m=i_{n-\m+1}$ for $\m=1,\dotsc,n-k-1$ maps $\cl{\D_{k+1}^{(I_{k+1})}}$ onto itself and
\begin{align}\label{eq_diffeom_dimwise}
\D_d^{(I_d)}\longmapsto\D_{k+1}^{(I_{k+1})}\times\cube_{d-k-1}^{(I_d\smin I_{k+1})}\quad\text{for $d=k+2,\dotsc,n$}
\end{align}
and altogether
\begin{align}\label{eq_diffeom_eventually}
\cl{\D_n^{(I_n)}}
\longmapsto
\Big(\cl{\D_{k+1}^{(I_{k+1})}}\times\cl{\cube_{n-k-1}^{(I_{n}\smin I_{k+1})}}\Big)\smin\bigcup_{j=k+1}^{n-1}N_{j}.
\end{align}
The $n-k-1$ additional $(n-1)$-dimen\-sional faces\index{cube!additional face}\sindex[not]{Nr@$N_\ri$} $N_{k+1},\dotsc,N_{n-1}$ of $\cl{\D_{k+1}^{(I_{k+1})}}\times\cl{\cube_{n-k-1}^{(I_{n}\smin I_{k+1})}}$ are given by
\begin{gather}\label{eq_new_face_2}
 N_{k+1}=\cl{\D_k^{(I_{k})}}\times\set{0}^{(\set{i_{k+1}})}\times\cl{\cubex_{n-k-1}^{(I_{n}\smin I_{k+1})}}\\
\intertext{and}
\label{eq_new_face_1}
 N_{j}=\cl{\D_{k+1}^{(I_{k+1})}}\times\cl{\cube_{j-k-2}^{(I_{j-1}\smin I_{k+1})}}\times\set{0}^{(\set{i_{j}})}
\times\cl{\cubex_{n-j}^{(I_{n}\smin I_{j})}}
\end{gather}
for $j=k+2,\dotsc,n-1$. 
Simultaneously, 
the operator $L^*=\sum p^i(\delta^i_j-p^j)\dd{}{p^i}\dd{}{p^j}$ in $\cl{\D_n^{(I_n)}}$ is transformed into%
\footnote{%
Please note that on boundary instances of $\cube_{n-k-1}^{(I_{n}\smin I_{k+1})}$, \ie $\tilde{p}^{i_l}=0$ for some $l\in I_n\smin I_{k+1}$, the corresponding summands are assumed not to appear in the right sum in equation~\eqref{eq_L_cube}, which may be interpreted as a result of a successive restriction. The given domain is the maximal domain for the operator as it is not defined on the exception set $\bigcup_{j=k+1}^{n-1}N_{j}$ (however, \cf also lemma~\ref{lem_stem_cube} for the stationary case).
}%
\begin{align}\label{eq_L_cube}
\tilde{L}^*=\half\sum_{j,l=1}^{k+1}{\tilde{p}}^{i_j}(\delta^j_l-{\tilde{p}}^{i_l})\dd{}{{\tilde{p}}^{i_j}}\dd{}{{\tilde{p}}^{i_l}}
+\half\sum_{j=k+2}^{n} \frac{\tilde{p}^{i_j}(1-\tilde{p}^{i_j})}{\prod_{l=k+1}^{j-1} \tilde{p}^{i_l}}\ddsq{}{\tilde{p}^{i_j}}
\end{align}
in $\Big(\cl{\D_{k+1}^{(I_{k+1})}}\times\cl{\cube_{n-k-1}^{(I_{n}\smin I_{k+1})}}\Big)\smin\bigcup_{j=k+1}^{n-1}N_{j}$.%

If in any step the coordinate $\tilde{p}^{\ri_j}$ is chosen with alternative orientation (\cf remark~\ref{rmk_blowup_trans}), all appearances of $\tilde{p}^{\ri_j}$ in the above formulae are replaced by $(1-\tilde{p}^{\ri_j})$.
\end{prop}

Thus, the iterated blow-up\index{blow-up transformation!iterated} translates the (extended) \index{Kolmob@\KBE!extended ($n$-dim)}\KBE in $\cl{\D}_n$ into a corresponding differential equation in  $\Big(\cl{\D_{k+1}^{(I_{k+1})}}\times\cl{\cube_{n-k-1}^{(I_{n}\smin I_{k+1})}}\Big)\smin\bigcup_{j=k+1}^{n-1}N_{j}$. For the successively extended solutions of the \KBE introduced in the preceding chapter, the transformation behaviour is as follows:

\begin{prop}\label{prop_blowup_solution}
Let $k,n\in\N$ with $0\leq k\leq n-2$, $\set{i_k,i_{k+1},\dotsc,i_n}\subset I_n\ce\set{0,1,\dotsc,n}$ with $i_i\neq i_j$ for $i\neq j$  and $I_d\ce I_n \smin\set{i_{d+1},\dotsc,i_n}$ for $d=k,\dotsc,n-1$, and let $u_{I_k}$ in $\big({\D_k^{(I_k)}}\big)_{-\infty}$  and $\bar{U}_{I_k}^{i_k,\dotsc,i_{n}}$ in $\Big(\bigcup_{ k \leq d\leq n}\D_d^{(I_d)}\Big)_{-\infty}$ as in proposition~\ref{prop_ext_iter}. Then a repeated {blow-up transformation}\index{blow-up transformation!iterated} $\Phi^{\si_{n-k-1}}_{\ri_{n-k-1}}\circ\dotso\circ\Phi^{\si_1}_{\ri_1}$ with $\Phi^{\si_\m}_{\ri_\m}$ as in lemma~\ref{lem_blowup_trans} with $\si_\m=i_{n-\m}$ and $\ri_\m=i_{n-\m+1}$ for $\m=1,\dotsc,n-k-1$ converts 
\begin{align}\notag
\bar{U}_{I_k}^{i_k,\dotsc,i_{n}}(p,t) &\ce u_{I_k}(p,t)\ind_{\D_k^{(I_k)}}(p)+\sum_{ k+1\leq d\leq n}\bar{u}_{I_k}^{i_k,\dotsc,i_{d}}(p,t)\ind_{\D_d^{(I_d)}}(p)\\
&=u_{I_k}(p,t)\ind_{\D_k^{(I_k)}}(p)+\sum_{ k+1\leq d\leq n} u_{I_k}(\pi^{i_k,\dotsc,i_d}(p),t)\prod_{j=k}^{d-1} \frac{p^{i_j}}{\sum_{l=j}^d p^{i_l}} \ind_{\D_d^{(I_d)}}(p)
\end{align}
%
\text{on $\Big(\bigcup_{ k \leq d\leq n}\D_d^{(I_d)}\Big)_{-\infty}$} into
\begin{multline}
\tilde{U}_{I_k}^{i_k,i_{k+1};i_{k+2},\dotsc,i_{n}}(\tilde{p},t)\ce
u_{I_k}(\tilde{p},t)\ind_{\D_k^{(I_k)}}(\tilde{p})\\
+\sum_{k+1\leq d\leq n}\tilde{u}_{I_k}^{i_k,i_{k+1};i_{k+2},\dotsc,i_{d}}(\tilde{p},t)\ind_{{\D_{k+1}^{(I_{k+1})}}\times\cube^{(I_d\smin I_{k+1})}_{d-k-1}}(\tilde{p})
\end{multline}
\text{on $\Big(\bigcup_{ k \leq d\leq n}{\D_{k+1}^{(I_{k+1})}}\times{\cube_{n-k-1}^{(I_{n}\smin I_{k+1})}}\Big)_{-\infty}$} with
\begin{align}
\tilde{u}_{I_k}^{i_k,i_{k+1};i_{k+2},\dotsc,i_{d}}(\tilde{p},t)
\ce\bar{u}^{i_k,i_{k+1}}_{I_k}(\tilde\pi^{i_{k+1}}(\tilde{p}),t)\prod_{j=k+2}^{d} (1-\tilde{p}^{i_j})\quad\text{for $d=k+2,\dotsc,n$}
\end{align}
with $\tilde{\pi}^{i_{k+1}}(\tilde{p}^{i_j})\ce\tilde{p}^{i_j}$ for $i_j\in I_{k+1}$, $\tilde{\pi}^{i_{k-1}}(\tilde{p}^{i_j})\ce 0$ else. The transformed functions
$\tilde{u}_{I_k}^{i_k,i_{k+1};i_{k+2},\dotsc,i_{d}}$ smoothly extend to $\Big(\cl{\D_{k+1}^{(I_{k+1})}}\times\cl{\cube^{(I_d\smin I_{k+1})}_{d-k-1}}\Big)_{-\infty}$ respectively; consequently  also $\tilde{U}_{I_k}^{i_k,i_{k+1};i_{k+2},\dotsc,i_{n}}$ smoothly extends to $\Big(\cl{\D_{k+1}^{(I_{k+1})}}\times\cl{\cube_{n-k-1}^{(I_{n}\smin I_{k+1})}}\Big)_{-\infty}$. 
Furthermore, it may be simplified to
\begin{align}\label{eq_blowup_sol_sim}
\tilde{U}_{I_k}^{i_k,i_{k+1};i_{k+2},\dotsc,i_{n}}(\tilde{p},t)\equiv
\tilde{u}_{I_k}^{i_k,i_{k+1};i_{k+2},\dotsc,i_{n}}(\tilde{p},t)
\quad\text{in $\Big({\cl{\D_{k+1}^{(I_{k+1})}}\times\cl{\cube^{(I_n\smin I_{k+1})}_{n-k-1}}}\Big)_{-\infty}$}.
\end{align}
If in any step the coordinate $\tilde{p}^{\ri_j}$ is chosen with alternative orientation (\cf remark~\ref{rmk_blowup_trans}), all appearances of $\tilde{p}^{\ri_j}$ in the above formulae need to be replaced with $(1-\tilde{p}^{\ri_j})$.
\end{prop}

For the stationary components, we have in particular:
\begin{cor}\label{cor_blowup_solution}
For $k=0$ and \wlg $i_0=0$, the transformed function of proposition~\ref{prop_blowup_solution} in equation~\eqref{eq_blowup_sol_sim} simplifies to 
 \begin{align}\label{eq_blowup_sol_sim_0}
&\tilde{U}_{\set{i_0}}^{i_0,i_{1};i_{2},\dotsc,i_{n}}(\tilde{p})=
u_{\set{i_0}}(1)\cdot
\prod_{j=1}^{n} (1-\tilde{p}^{i_j})\quad\text{in $\cl{\cube_{n\phantom{1}}^{(I_n^\prime)}}$,}
\end{align}
while in accordance with proposition~\ref{prop_blowup_trans_total} 
%
the domain is mapped 
\begin{align}\label{eq_diffeom_dimwise_0}
\D_d^{(I_d)}\longmapsto\cube_{d}^{(I^\prime_d)}\quad\text{for $d=0,\dotsc,n$}
\end{align}
and altogether 
\begin{align}\label{eq_diffeom_eventually_0}
\cl{\D_n^{(I_n)}}
\longmapsto
\cl{\cube_{n\phantom{1}}^{(I_n^\prime)}}\smin\bigcup_{j=1}^{n-1}N_{j}.
\end{align}
The $n-1$ additional $(n-1)$-dimen\-sional faces $N_{1},\dotsc,N_{n-1}$ of $\bd{\cube_{n\phantom{1}}^{(I_n^\prime)}}$ are given by
\begin{gather}\label{eq_new_face_1_0}		
N_{1}=\set{0}^{(\set{i_{1}})}
\times\cl{\cubex_{n-1}^{(I_{n}^\prime\smin I^\prime_{1})}}\\
\intertext{and}
N_{j}=\cl{\cube_{j-1}^{(I^\prime_{j-1})}}\times\set{0}^{(\set{i_{j}})}
\times\cl{\cubex_{n-j}^{(I_{n}^\prime\smin I^\prime_{j})}}
\end{gather}
for $j=2,\dotsc,n-1$, 
whereas the operator $L^*=\sum p^i(\delta^i_j-p^j)\dd{}{p^i}\dd{}{p^j}$ in $\cl{\D_n^{(I_n)}}$
 is transformed into
\begin{align}\label{eq_L_cube_0}
\tilde{L}^*=\half\sum_{j=1}^{n} \frac{\tilde{p}^{i_j}(1-\tilde{p}^{i_j})}{\prod_{l=1}^{j-1} \tilde{p}^{i_l}}\ddsq{}{\tilde{p}^{i_j}}
\quad\text{in $\cl{\cube_{n\phantom{1}}^{(I_n^\prime)}}\smin\bigcup_{j=1}^{n-1}N_{j}$.}
\end{align}
\end{cor}

\begin{proof}[Proof of propositions~\ref{prop_blowup_trans_total} and \ref{prop_blowup_solution}]


We prove the assertions of both propositions in parallel: Aiming to transform $\bar{U}_{I_k}^{i_k,\dotsc,i_{n}}$ into a function that does not feature any incompatibilities and hence is of sufficient regularity with respect to the entire closure of the (transformed) domain, we show that the full blow-up via a repeated application of the coordinate transformation $\Phi^\si_\ri$ of lemma~\ref{lem_blowup_trans} with the indices $\si$ and $\ri$ to be picked as shown in each step yields the desired result for $\tilde{U}_{I_k}^{i_k,i_{k+1};i_{k+2},\dotsc,i_{n}}$, 
whereas the transformation behaviour of the domain and the operator 
is as stated in proposition~\ref{prop_blowup_trans_total}. 
Note that in the designation of any domains, we will usually suppress the $t$-component throughout this proof, \eg write $\D_n^{(I_n)}$ instead of $\big(\D_n^{(I_n)}\big)_{-\infty}$, for notational simplicity.

Starting with the top-dimen\-sional component of $\bar{U}_{I_k}^{i_k,\dotsc,i_{n}}$, which is
\begin{align}\notag
 \bar{u}_{I_k}^{i_k,\dotsc,i_{n}}(p,t)
&=\bar{u}_{I_k}^{i_k,\dotsc,i_{n-1}}(\pi^{i_{n-1},i_n}(p),t)\cdot \frac{p^{i_{n-1}}}{p^{i_{n-1}}+ p^{i_n}}\\
&=u_{I_k}(\pi^{i_k,\dotsc,i_{n-1}}(\pi^{i_{n-1},i_n}(p)),t)\prod_{j=k}^{n-2} \frac{p^{i_j}}{\sum_{l=j}^n p^{i_l}}\cdot \frac{p^{i_{n-1}}}{p^{i_{n-1}}+ p^{i_n}}\quad\text{in $\D_n^{(I_n)}$}
\end{align}
with $p^{i_0}\equiv p^0=1-\sum_{j=1}^np^{i_j}$ (if $i_0\neq 0$, one may change the coordinates, \ie permute the vertices correspondingly),
we initially put%
\footnote{%
Alternatively, one could also put $\si_1\ce i_{n}$ and $\ri_1\ce i_{n-1}$, which would correspond to inverting the orientation of the $\tilde{p}^{\ri_1}$-coordinate in accordance with remark~\ref{rmk_blowup_trans} (\cf also below) plus subsequently swapping the coordinate indices $i_n$ and $i_{n-1}$, thus $\tilde{p}^{i_{n}}$ would get replaced with $1-\tilde{p}^{i_{n-1}}$ and $\tilde{p}^{i_{n-1}}$ with~$\tilde{p}^{i_{n}}$.}\label{footnote_swap_ind} 
 $\si_1\ce i_{n-1}$ and $\ri_1\ce i_{n}$. 
Changing coordinates $(p^i)\mapsto(\tilde{p}^i)$ by $\Phi^{\si_1}_{\ri_1}$ maps $\D_n^{(I_n)}$ onto $\D_{n-1}^{(I_{n-1})}\times{\cube_{1}^{(\set{i_{n}})}}$ and  $\cl{\D_{n-1}^{(I_{n-1})}}$ onto $\cl{\D_{n-1}^{(I_{n-1})}}\times\set{0}^{(\set{i_{n}})}$, whereas the entire domain 
$\cl{\D_n^{(I_n)}}$ is transformed into $\Big(\cl{\D_{n-1}^{(I_{n-1})}}\times\cl{\cube_{1}^{(\set{i_{n}})}}\Big)\smin N_{n-1}$ with 
\begin{align}
& N_{n-1}\ce\cl{\D_{n-2}^{(I_{n-2})}}\times\set{0}^{(\set{i_{n-1}})}\times\cl{\cubex_1^{(\set{i_{n}})}}
\end{align}
being an additional $(n-1)$-dimen\-sional face of $\cl{\D_{n-1}^{(I_{n-1})}}\times\cl{\cube_{1}^{(\set{i_{n}})}}$ (\cf lemma~\ref{lem_blowup_trans}).
Simultaneously, the $(n-2)$-dimen\-sional incompatibility at $\cl{\D_{n-2}^{(I_{n-2})}}$ of the continuous extension of $\bar{u}_{I_k}^{i_k,\dotsc,i_{n}}$ to $\bd_{n-1}{\D_n^{(I_{n})}}$ is removed as the transformation yields
\begin{multline}
\begin{aligned}
\tilde{u}_{I_k}^{i_k,\dotsc,i_{n-1};i_{n}}(\tilde{p},t)
&\ce\bar{u}_{I_k}^{i_k,\dotsc,i_{n-1}}(\tilde{\pi}^{i_{n-1}}(\tilde{p}),t)\cdot(1-\tilde{p}^{i_{n}})\\
&=u_{I_k}(\pi^{i_k,\dotsc,i_{n-1}}(\tilde{\pi}^{i_{n-1}}(\tilde{p})),t)\prod_{j=k}^{n-2} \frac{\tilde{p}^{i_j}}{\sum_{l=j}^n \tilde{p}^{i_l}}
\cdot(1-\tilde{p}^{i_{n}})
\end{aligned}\\
\text{in $\D_{n-1}^{(I_{n-1})}\times{\cube_{1}^{(\set{i_{n}})}}$}
\end{multline}
by equation~\eqref{eq_blowup_inverse} et seq.\ (note $\tilde{\pi}^{i_{n-1}}(\tilde{p})=\pi^{i_{n-1},i_n}(p)$).
Hence, the complete function $\bar{U}_{I_k}^{i_k,\dotsc,i_{n}}$ is transformed into
\begin{multline}
\tilde{U}_{I_k}^{i_k,\dotsc,i_{n-1};i_{n}}(p,t)
\ce\sum_{k\leq d\leq n-1}\bar{u}_{I_k}^{i_k,\dotsc,i_{d}}(p,t)\ind_{\D_d^{(I_d)}}(p)\\
+\tilde{u}_{I_k}^{i_k,\dotsc,i_{n-1};i_{n}}(p,t)\ind_{{\D_{n-1}^{(I_{n-1})}}\times\cube^{(I_n\smin I_{n-1})}_{1}}(p)
\end{multline}
with the transformed top-dimen\-sional component $\tilde{u}_{I_k}^{i_k,\dotsc,i_{n-1};i_{n}}(\tilde{p},t)$ smoothly extending to $\D_{n-1}^{(I_{n-1})}\times\cl{\cube_{1}^{(\set{i_{n}})}}$ with
\begin{align} 
 \tilde{u}_{I_k}^{i_k,\dotsc,i_{n-1};i_{n}}(\tilde{p},t)\big|_{\D_{n-1}^{(I_{n-1})}\times{\set{0}^{(\set{i_{n}})}}}\equiv\bar{u}_{I_k}^{i_k,\dotsc,i_{n-1}}(\tilde{p},t)\quad\text{in $\D_{n-1}^{(I_{n-1})}\times{\set{0}^{(\set{i_{n}})}}$}. 
\end{align}
As $\bar{u}_{I_k}^{i_k,\dotsc,i_{n-1}}$ itself smoothly extends to $\bd_{n-2}\D_{n-1}^{(I_{n-1})}$, thus $\tilde{u}_{I_k}^{i_k,\dotsc,i_{n-1};i_{n}}$ now smoothly extends to the entire  
$(\bd_{n-2}\D_{n-1}^{(I_{n-1})})\times\cl{\cube_{1}^{(\set{i_{n}})}}$, in particular to $\D_{n-2}^{(I_{n-2})}\times\cl{\cube_{1}^{(\set{i_{n}})}}\subset \cl{N}_{n-1}$ (however, $\bar{u}_{I_k}^{i_k,\dotsc,i_{n-1}}$ \resp its continuous extension to $\bd_{n-2}\D_{n-1}^{(I_{n-1})}$ still has an incompatibility at $\cl{\D_{n-3}^{(I_{n-3})}}$). 


The operator $L^*=\half\sum_{i,j=1}^n p^i(\delta^i_j-p^j)\dd{}{p^i}\dd{}{p^j}$ in $\cl{\D_n^{(I_n)}}$ transforms into (\cf lemma~\ref{lem_blowup_trans_op})
\begin{align}
\tilde{L}^*=\half\sum_{j,l\neq {n}} \tilde{p}^{i_j}(\delta^j_l-\tilde{p}^{i_l})\dd{}{\tilde{p}^{i_j}}\dd{}{\tilde{p}^{i_l}}
+\half\frac{\tilde{p}^{i_{n}}(1-\tilde{p}^{i_{n}})}{\tilde{p}^{i_{n-1}}}\dd{}{\tilde{p}^{i_{n}}}\dd{}{\tilde{p}^{i_{n}}}
\end{align}
on $\Big({\cl{\D_{n-1}^{(I_{n-1})}}}\times\cl{\cube_{1}^{(\set{i_{n}})}}\Big)\smin N_{n-1}$ since we have $\tilde{a}^{kl}(\tilde{p})=p^k(\delta^k_l-p^l)=\tilde{p}^k(\delta^k_l-\tilde{p}^l)$ for $k,l\neq i_{n-1},i_{n}$. 
If $\tilde{p}^{i_{n}}$ is chosen with alternative orientation (\cf remark~\ref{rmk_blowup_trans}), then $\tilde{p}^{i_{n}}$ needs to be replaced by $(1-\tilde{p}^{i_{n}})$ everywhere.

\bigskip

As already indicated, the transformed solution is still not smoothly extendable to the full boundary of the transformed domain: Its $(n-2)$-dimen\-sional incompatibility is resolved, but its lower-dimen\-sional incompatibilities persist. Thus, the highest-dimen\-sional incompatibility now is of dimension $n-3$
, and hence the situation is ready for another application of the blow-up transformation\index{blow-up transformation!iterated}, yielding a corresponding situation afterwards. 

Thus, an iterative advancement is necessary to resolve all incompatibilities. For this purpose, we assume that after the $\m$\ord step ($\m=1,\dotsc,n-k-2$) an already transformed function $\tilde{U}_{I_k}^{i_k,\dotsc,i_{n-\m};i_{n-\m+1},\dotsc,i_n}$ with (note that we again associate coordinates~$p$ resp.\ $\tilde{p}$
etc.\ to the domain before/after the $(\m+1)$\ord transition; furthermore, we will use the convention $\bar{u}^{i_k}_{I_k}\equiv u_{I_k} $ to simplify the notation)
%
%
\begin{multline}
\tilde{U}_{I_k}^{i_k,\dotsc,i_{n-\m};i_{n-\m+1},\dotsc,i_n}(p,t)
=\sum_{k\leq d\leq n-\m}\bar{u}_{I_k}^{i_k,\dotsc,i_{d}}(p,t)\ind_{\D_d^{(I_d)}}(p)\\
+\sum_{n-\m+1\leq d\leq n}\tilde{u}_{I_k}^{i_k,\dotsc,i_{n-\m};i_{n-\m+1},\dotsc,i_d}(p,t)\ind_{{\D_{n-\m}^{(I_{n-\m})}}\times\cube^{(I_d\smin I_{n-\m})}_{d-n+\m}}(p)
\end{multline}%
with
\begin{align}
\tilde{u}_{I_k}^{i_k,\dotsc,i_{n-\m};i_{n-\m+1},\dotsc,i_d}(p,t)=\bar{u}_{I_k}^{i_k,\dotsc,i_{n-\m}}(\tilde{\pi}^{i_{n-\m}}(p),t)\prod_{j=n-\m+1}^{d} (1-p^{i_j})
\end{align}
for $d=n-\m+1,\dotsc,n$ and
\begin{align}\notag
\bar{u}_{I_k}^{i_k,\dotsc,i_{n-\m}}(p,t)
&=\bar{u}_{I_k}^{i_k,\dotsc,i_{n-\m-1}}(\pi^{i_{n-\m-1},i_{n-\m}}(p),t)\cdot \frac{p^{i_{n-\m-1}}}{p^{i_{n-\m-1}}+ p^{i_{n-\m}}}\\
&=u_{I_k}(\pi^{i_k,\dotsc,i_{n-\m-1}}(\pi^{i_{n-\m-1},i_{n-\m}}(p)),t)\hspace*{-1pt}\prod_{j=k}^{n-\m-2} \frac{p^{i_j}}{\sum_{l=j}^n p^{i_l}}\cdot \frac{p^{i_{n-\m-1}}}{p^{i_{n-\m-1}}+ p^{i_{n-\m}}}
\end{align}
in $\D_{n-\m}^{(I_{n-\m})}$. The corresponding total domain as an image of $\cl{\D_n^{(I_n)}}$ is given by 
\begin{align}\label{eq_blowup_dom_prev}
\big(\cl{\D_{n-\m}^{(I_{n-\m})}}\times\cl{\cube_{\m}^{(I_{n}\smin I_{n-\m})}}\big)\smin\bigcup_{j=n-\m}^{n-1}N_{j}
\end{align}
with previously additional $(n-1)$-dimen\-sional faces 
\begin{gather}
 N_{n-\m}=\cl{\D_{n-\m-1}^{(I_{n-\m-1})}}\times\set{0}^{(\set{i_{n-\m}})}\times\cl{\cubex_{\m}^{(I_{n}\smin I_{n-\m})}}\\
 \intertext{and}
 N_{j}=\cl{\D_{n-\m}^{(I_{n-\m})}}\times\cl{\cube_{j-n+\m-1}^{(I_{j-1}\smin I_{n-\m})}}\times\set{0}^{(\set{i_{j}})}
\times\cl{\cubex_{n-j}^{(I_{n}\smin I_{j})}}
\end{gather}
for $j=n-\m+1,\dotsc,n-1$. 

The functions $\tilde{u}_{I_k}^{i_k,\dotsc,i_{n-\m};i_{n-\m+1},\dotsc,i_d}$ smoothly extend each to ${\D_{n-\m}^{(I_{n-\m})}}\times\cl{\cube_{d-n+\m}^{(I_d\smin I_{n-\m})}}$,  
and we have
\begin{align}
\tilde{u}_{I_k}^{i_k,\dotsc,i_{n-\m};i_{n-\m+1},\dotsc,i_d}|_{{{\D_{n-\m}^{(I_{n-\m})}}\times\cube^{(I_{d-1}\smin I_{n-\m})}_{\m}}}=
\tilde{u}_{I_k}^{i_k,\dotsc,i_{n-\m};i_{n-\m+1},\dotsc,i_{d-1}}
\end{align}
for $d=n-\m+2,\dotsc,n$ and 
\begin{align}
\tilde{u}_{I_k}^{i_k,\dotsc,i_{n-\m};i_{n-\m+1}}|_{{{\D_{n-\m}^{(I_{n-\m})}}}}=
\bar{u}_{I_k}^{i_k,\dotsc,i_{n-\m}}.
\end{align}
With $\bar{u}_{I_k}^{i_k,\dotsc,i_{n-\m}}$ being smoothly extendable to $\bd_{n-\m-1}\D_{n-\m}^{(I_{n-\m})}$, also the functions $\tilde{u}_{I_k}^{i_k,\dotsc,i_{n-\m};i_{n-\m+1},\dotsc,i_d}$ smoothly extend to  $\Big(\bd_{n-\m-1}{\D_{n-\m}^{(I_{n-\m})}}\Big)\times\cl{\cube_{d-n+\m}^{(I_d\smin I_{n-\m})}}$, in particular all additional faces\index{cube!additional face} are covered. 

Furthermore, we assume the operator $L^*$ to be of the corresponding form
\begin{align}\label{eq_L_previous}
{L}^*=\half\sum_{j,l=1}^{n-\m}{p}^{i_j}(\delta^j_l-{p}^{i_l})\dd{}{{p}^{i_j}}\dd{}{{p}^{i_l}}
+\half\sum_{j=n-\m+1}^{n} \frac{p^{i_j}(1-p^{i_j})}{\prod_{l=n-\m}^{j-1} p^{i_l}}\ddsq{}{p^{i_j}}
\end{align}
on $\Big(\cl{\D_{n-\m}^{(I_{n-\m})}}\times\cl{\cube_{\m}^{(I_{n}\smin I_{n-\m})}}\Big)\smin\bigcup_{j=n-\m}^{n-1}N_{j}$.


For the $(\m+1)$\ord blow-up step going to be applied now, we first notice that $\bar{u}_{I_k}^{i_k,\dotsc,i_{n-\m}}$ \resp its continuous extension to $\bd_{n-\m-1}\D_{n-\m}^{(I_{n-\m})}$ still has an incompatibility at $\cl{\D_{n-\m-2}^{(I_{n-\m-2})}}\subset\D_{n-\m}^{(I_{n-\m})}$, corresponding to $p^{i_{n-\m}}+p^{i_{n-\m-1}}=0$. 
Consequently, this may be resolved by a blow-up transformation $\Phi^{\si_{\m+1}}_{\ri_{\m+1}}$ with $\si_{\m+1}=i_{n-\m-1}$ and $\ri_{\m+1}=i_{n-\m}$ (note that, due to the stipulation $i_0=0$, we always have $\si_{\m+1},\ri_{\m+1}\neq 0$), mapping the simplex part of the domain (\cf lemma~\ref{lem_blowup_trans})
\begin{align}\label{eq_blowup_inner_dom}
{\D_{n-\m}^{(I_{n-\m})}}
\longmapsto{\D_{n-\m-1}^{(I_{n-\m-1})}}\times{\cube_1^{(\set{i_{n-\m}})}}
\end{align}
\resp
\begin{align}
\cl{\D_{n-\m-1}^{(I_{n-\m-1})}}
\longmapsto\cl{\D_{n-\m-1}^{(I_{n-\m-1})}}\times{\set{0}^{(\set{i_{n-\m}})}}
\end{align}
and altogether
\begin{align}\label{eq_blowup_outer_dom}
\cl{\D_{n-\m}^{(I_{n-\m})}}\longmapsto\cl{\D_{n-\m-1}^{(I_{n-\m-1})}}\times\cl{\cube_1^{(\set{i_{n-\m}})}}\smin N_{n-\m-1}
\end{align}
with
\begin{align}\label{eq_new_face_before}
N_{n-\m-1}\ce\cl{\D_{n-\m-2}^{(I_{n-\m-2})}}\times\set{0}^{(\set{i_{n-\m-1}})}\times \cl{\cubex_1^{(\set{i_{n-\m}})}}
\end{align}
being an additional $(n-\m-1)$-dimen\-sional face of $\cl{\D_{n-\m-1}^{(I_{n-\m-1})}}\times\cl{\cube_1^{(\set{i_{n-\m}})}}$. 

From this -- when gradually adding the cubic part ${\cube_{\m}^{(I_{n}\smin I_{n-\m})}}$ with coordinates $p^{i_{n-\m+1}},\dotsc,p^{i_{n}}$ -- equation~\eqref{eq_blowup_inner_dom} turns into
\begin{align}
{\D_{n-\m}^{(I_{n-\m})}}\times\cube_{d-n+\m}^{(I_{d}\smin I_{n-\m})}
\longmapsto{\D_{n-\m-1}^{(I_{n-\m-1})}}\times\cube_{d-n+\m+1}^{(I_{d}\smin I_{n-\m-1})}\quad\text{for $d\geq n-\m$},
\end{align}
and by applying equation~\eqref{eq_blowup_outer_dom} to the previous image of the initial domain $\cl{\D_n^{(I_n)}}$ in equation~\eqref{eq_blowup_dom_prev}, we obtain for the transformed total domain
\begin{align}\label{eq_blowup_dom_after}
\big(\cl{\D_{n-\m-1}^{(I_{n-\m-1})}}\times\cl{\cube_{\m+1}^{(I_{n}\smin I_{n-\m-1})}}\big)\smin\bigcup_{j=n-\m-1}^{n-1}\tilde{N}_{j}
\end{align}
with $\tilde{N}_{n-\m},\dotsc,\tilde{N}_{n-1}$ being the images of the previous\index{cube!additional face} additional faces: The faces $N_{n-\m+1},\dotsc,N_{n-1}$ are only affected indirectly as they contain the full $\cl{\D_{n-\m}^{(I_{n-\m})}}$ as a factor, and hence only the $i_{n-\m}$\ord coordinate is moved from the simplex to the cubic fraction, thus
\begin{align}
& \tilde{N}_{j}=\cl{\D_{n-\m-1}^{(I_{n-\m-1})}}\times\cl{\cube_{j-n+\m}^{(I_{j-1}\smin I_{n-\m-1})}}\times\set{0}^{(\set{i_{j}})}
\times\cl{\cubex_{n-j}^{(I_{n}\smin I_{j})}}
\end{align}
\text{for $j=n-\m+1,\dotsc,n-1$},  
whereas $N_{n-\m}\equiv\tilde{N}_{n-\m}$ is virtually not affected as only ${p}^{i_{n-\m}}=0$ is transformed into $\tilde{p}^{i_{n-\m}}=0$.
For the `new' additional $(n-1)$-dimen\-sional face $\tilde{N}_{n-\m-1}$
(resulting from $N_{n-\m-1}$), we may -- having added the remaining dimensions -- relax the condition $\tilde{p}^{i_{n-\m}}>0$ in equation~\eqref{eq_new_face_before}, which ensures $N_{n-\m-1}\neq\cl{\D_{n-\m-2}^{(I_{n-\m-2})}}$, into $\sum_{j=n-\m}^n\tilde{p}^{i_{j}}>0$ and hence obtain
\begin{align}
& \tilde{N}_{n-\m-1}\ce\cl{\D_{n-\m-2}^{(I_{n-\m-2})}}\times\set{0}^{(\set{i_{n-\m-1}})}
\times\cl{\cubex_{\m+1}^{(I_{n}\smin I_{n-\m-1})}}.
\end{align}

Simultaneously, $\bar{u}_{I_k}^{i_k,\dotsc,i_{n-\m}}$ and $\tilde{u}_{I_k}^{i_k,\dotsc,i_{n-\m};i_{n-\m+1},\dotsc,i_d}$, $d=n-\m+1,\dotsc,n$ get transformed into 
\begin{align}
\tilde{u}_{I_k}^{i_k,\dotsc,i_{n-\m-1};i_{n-\m},\dotsc,i_d}(\tilde{p},t)=\bar{u}_{I_k}^{i_k,\dotsc,i_{n-\m-1}}(\tilde{\pi}^{i_{n-\m-1}}(\tilde{p}),t)\prod_{j=n-\m}^{d} (1-\tilde{p}^{i_j}) 
\end{align}
in ${\D_{n-\m-1}^{(I_{n-\m-1})}}\times\cube_{d-n+\m+1}^{(I_{d}\smin I_{n-\m-1})}$ for $d\geq n-\m$,
and hence
\begin{multline}
\tilde{U}_{I_k}^{i_k,\dotsc,i_{n-\m-1};i_{n-\m},\dotsc,i_n}(p,t)\ce
\sum_{k\leq d\leq n-\m-1}\bar{u}_{I_k}^{i_k,\dotsc,i_{d}}(p,t)\ind_{\D_d^{(I_d)}}(p)\\
+\sum_{n-\m\leq d\leq n}\tilde{u}_{I_k}^{i_k,\dotsc,i_{n-\m-1};i_{n-\m},\dotsc,i_d}(p,t)\ind_{{\D_{n-\m-1}^{(I_{n-\m-1})}}\times\cube^{(I_d\smin I_{n-\m-1})}_{d-n+\m+1}}(p).
\end{multline}
%

The transformed functions $\tilde{u}_{I_k}^{i_k,\dotsc,i_{n-\m-1};i_{n-\m},\dotsc,i_d}$ then each smoothly extend to ${\D_{n-\m-1}^{(I_{n-\m-1})}}\times\cl{\cube_{d-n+\m+1}^{(I_d\smin I_{n-\m-1})}}$, 
and we have
\begin{align}\label{eq_restr_1}
\tilde{u}_{I_k}^{i_k,\dotsc,i_{n-\m-1};i_{n-\m},\dotsc,i_d}|_{{{\D_{n-\m-1}^{(I_{n-\m-1})}}\times\cube^{(I_{d-1}\smin I_{n-\m-1})}_{\m+1}}}=
\tilde{u}_{I_k}^{i_k,\dotsc,i_{n-\m-1};i_{n-\m},\dotsc,i_{d-1}}
\end{align}
for $d=n-\m+1,\dotsc,n$ and 
\begin{align}\label{eq_restr_2}
\tilde{u}_{I_k}^{i_k,\dotsc,i_{n-\m-1};i_{n-\m}}|_{{{\D_{n-\m-1}^{(I_{n-\m-1})}}}}=
\bar{u}_{I_k}^{i_k,\dotsc,i_{n-\m-1}}.
\end{align}
With $\bar{u}_{I_k}^{i_k,\dotsc,i_{n-\m-1}}$ being smoothly extendable to $\bd_{n-\m-2}\D_{n-\m-1}^{(I_{n-\m-1})}$, the functions $\tilde{u}_{I_k}^{i_k,\dotsc,i_{n-\m-1};i_{n-\m},\dotsc,i_d}$ also smoothly extend to $\Big(\bd_{n-\m-2}{\D_{n-\m-1}^{(I_{n-\m-1})}}\Big)\times\cl{\cube_{d-n+\m+1}^{(I_d\smin I_{n-\m-1})}}$, by which all additional faces\index{cube!additional face} are covered; in particular, $\tilde{u}_{I_k}^{i_k,\dotsc,i_{n-\m-1};i_{n-\m}}$ smoothly extends to $N_{n-\m-1}$ \resp eventually  $\tilde{u}_{I_k}^{i_k,\dotsc,i_{n-\m-1};i_{n-\m},\dotsc,i_n}$ extends to $\tilde{N}_{n-\m-1}$ 
(however, $\bar{u}_{I_k}^{i_k,\dotsc,i_{n-\m-1}}$ \resp its continuous extension to $\bd_{n-\m-2}\D_{n-\m-1}^{(I_{n-\m-1})}$ still has an incompatibility at $\cl{\D_{n-\m-3}^{(I_{n-\m-3})}}$).


To analyse the transformation behaviour of the operator, we first note that the requirements of lemma~\ref{lem_blowup_trans_op} on $a^{ij}$ are met as for $i,j\in\set{i_1,\dotsc,i_{n-\m}}$ we have $a^{ij}(p)=p^i(\delta^i_j-p^j)$ by equation~\eqref{eq_L_previous}, while all other non-diagonal coefficients vanish. Hence, by the lemma, we have for $i,j\in\set{i_1,\dotsc,i_{n-\m}}$
\begin{align}
\tilde{a}^{ij}(\tilde{p})=
\tilde{p}^i(\delta^i_j-\tilde{p}^{j}),
\end{align}
while for $\tilde{a}^{i_ji_j}$ with $j=n-\m+1,\dotsc,n$ we obtain
\begin{align}
\tilde{a}^{i_ji_j}(\tilde{p})=a^{i_ji_j}(p)=\frac{p^{i_j}(1-p^{i_j})}{\prod_{l=n-\m}^{j-1} p^{i_l}}
=\frac{\tilde{p}^{i_j}(1-\tilde{p}^{i_j})}{\prod_{l=n-\m-1}^{j-1} \tilde{p}^{i_l}}.
\end{align}
Likewise, $\tilde{a}^{i_{n-\m}i_{n-\m}}$ takes the form
\begin{align}
\tilde{a}^{i_{n-\m}i_{n-\m}}(\tilde{p})=\frac{\tilde{p}^{i_{n-\m}}(1-\tilde{p}^{i_{n-\m}})}{\tilde{p}^{i_{n-\m-1}}},
\end{align}
whereas all other coefficients vanish. Altogether, this yields
\begin{align}
\tilde{L}^*=\half\sum_{j,l=1}^{n-\m-1} \tilde{p}^{i_j}(\delta^j_l-\tilde{p}^{i_l})\dd{}{\tilde{p}^{i_j}}\dd{}{\tilde{p}^{i_l}}
+\half\sum_{j=n-\m}^{n} \frac{\tilde{p}^{i_j}(1-\tilde{p}^{i_j})}{\prod_{l=n-\m-1}^{j-1} \tilde{p}^{i_l}}\ddsq{}{\tilde{p}^{i_j}}
\end{align}
on  $\Big(\cl{\D_{n-\m-1}^{(I_{n-\m-1})}}\times\cl{\cube_{\m+1}^{(I_{n}\smin I_{n-\m-1})}}\Big)\smin\bigcup_{j=n-\m-1}^{n-1}N_{j}$. If $\tilde{p}^{i_{n-\m}}$ is chosen with alternative orientation (\cf remark~\ref{rmk_blowup_trans}), then $\tilde{p}^{i_{n-\m}}$ needs to be replaced by $(1-\tilde{p}^{i_{n-\m}})$ everywhere.

Thus, after the $(\m+1)$\ord blow-up step, domain, solution and operator are of analogous form as before, just with the index~$\m$ replaced by $\m+1$. Eventually, after $n-k-1$ blow-up steps domain, solution and operator have attained the asserted form of the corresponding statements. In particular, the remaining $u_{I_k}$ as a proper solution smoothly extends to the entire boundary of $\D_k^{(I_k)}$, and hence so does $\bar{u}_{I_k}^{i_k,i_{k+1}}$ in $\D_{k+1}^{(I_k)}$, implying that each
$\tilde{u}_{I_k}^{i_k,i_{k+1};i_{k+2},\dotsc,i_{d}}$ smoothly extends to $\cl{\D_{k+1}^{(I_{k+1})}}\times\cl{\cube^{(I_d\smin I_{k+1})}_{d-k-1}}$, 
and eventually $\tilde{U}_{I_k}^{i_k,i_{k+1};i_{k+2},\dotsc,i_{n}}$  smoothly extends to $\cl{\D_{k+1}^{(I_{k+1})}}\times\cl{\cube_{n-k-1}^{(I_{n}\smin I_{k+1})}}$. Moreover, the restriction property in equations~\eqref{eq_restr_1} and~\eqref{eq_restr_2} yields equation~\eqref{eq_blowup_sol_sim}.
\end{proof}

\begin{proof}[Proof of corollary~\ref{cor_blowup_solution}]
In the given setting, we have $\bar{u}_{\set{i_0}}^{i_0,i_{1}}(\tilde{p})\hspace*{-0.2pt}=\hspace*{-0.2pt}u_{\set{i_0}}(\tilde{p}^{i_0}+\tilde{p}^{i_1})\frac{\tilde{p}^{i_0}}{\tilde{p}^{i_0}+\tilde{p}^{i_1}}=u_{\set{i_0}}(1)(1-\tilde{p}^{i_1})$ in $\cl{\D_1^{(\set{i_0,i_1})}}=\cl{\cube_1^{(\set{i_1})}}$ (and $\cl{\D_0^{(\set{i_0})}}=\set{0}^{(\set{i_0})}$), which proves the asserted form of the (simplified) solution, the domain and the additional faces\index{cube!additional face}.
\end{proof}

However, the global smoothness of the transformed solution of proposition~\ref{prop_ext_iter} observed in the preceding corollary does not necessarily hold for other functions in question, \ie arbitrary iteratively extended solutions\index{backward extension!iterated} $U$ in accordance with the extension constraints\index{extension constraints (n-dim)@extension constraints ($n$-dim)}~\ref{dfi_ext} (this corresponds to $U$ particularly being of class $C_{p_0}^\infty$). 
However, we still have a weaker global regularity assertion for the transformed function $\tilde{U}$ on the entire image of the simplex (only formulated for the stationary component corresponding to the setting of corollary~\ref{cor_blowup_solution}):

\begin{lem}\label{lem_blowup_reg} 
Let $n\geq 2$, $I_d\ce\set{i_0,i_{1},\dotsc,i_d}\subset \set{0,1,\dotsc,n}$ for $d=0,\dotsc,n$ with $i_i\neq i_j$ for $i\neq j$ and $u_{\set{i_0}}\co \D_0^{(\set{i_0})}\map \R$. 
Then an iterated extension\index{backward extension!iterated} ${U}=\sum_{d=0}^n {u}_d \in C_{p_0}^\infty\big(\bigcup_{d=0}^n \D_{d}^{(I_d)}\big)$ of $u_{\set{i_0}}$ in accordance with the extension constraints~\ref{dfi_ext} is transformed by a successive {blow-up transformation}\index{blow-up transformation!iterated} $\Phi^{\si_{n-1}}_{\ri_{n-1}}\circ\dotso\circ\Phi^{\si_1}_{\ri_1}$
as in proposition~\ref{prop_blowup_trans_total} 
into a function 
$\tilde{U}=\sum_{d=0}^n \tilde{u}_d\co\bigcup_{d=0}^n \cube_{d}^{(I_d^\prime)}\map\R$ with extension to all faces $\bigset{\tilde{p}^{i_1}=1},\dotsc,
\bigset{\tilde{p}^{i_n}=1}$ (perceivable as boundary instance of any $\cube_{d}^{(I_d^\prime)}\subset\cl{\cube_n^{(I^{\prime}_n)}}$) which is of class $C_p^\infty$ and vanishes on the mentioned faces.
\end{lem}

For the proof, we trace the extendability of $\tilde{U}$ towards the additional faces\index{cube!additional face} back to that of~$U$ in $\cl{\D_n^{(I_n)}}$ for approaching the incompatibilities -- which is accomplished by the priorly following lemma. Note that in the following we will use a disjoint formulation of the additional faces by putting
\begin{align}
N_{j}={\cube_{j-1}^{(I_{j-1}^\prime)}}\times\set{0}^{(\set{i_{j}})}
\times\cl{\cubex_{n-j}^{(I_{n}^\prime\smin I_{j}^\prime)}}.
\end{align}

\begin{lem}\label{lem_seq_N}
In the setting of a full {blow-up transformation}\index{blow-up transformation!iterated} as in proposition~\ref{prop_blowup_trans_total}, for $d=1,\dotsc,n$ the additional face $N_d ={\cube_{d-1}^{(I_{d-1}^\prime)}}\times\set{0}^{(\set{i_{d}})}
\times\cl{\cubex_{n-d}^{(I_{n}^\prime\smin I_{d}^\prime)}}\subset\overline{\cube_n^{(I_{n}^\prime)}}$ corresponds to $\D_{d-1}^{(I_{d-1})}\subset\overline{\D^{(I_{n})}_n}$ with additional values existing for $\frac{p^{i_{d+1}}+\dotsc+p^{i_{n}}}{p^{i_{d}}+p^{i_{d+1}}+\dotsc+p^{i_{n}}},\dotsc,\frac{p^{i_n}}{p^{i_{n-1}}+p^{i_n}}$ (perceivable as limits of corresponding sequences).
Furthermore, for $j=1,\dotsc,d-1$ the face $\set{\tilde{p}^{i_j}=1}\subset\overline{\cube_{d-1}^{(I_{d-1}^\prime)}}$ corresponds to ${p}^{i_{j-1}}=0$ in $\overline{\D_{d-1}^{(I_{d-1})}}$, in particular its interior corresponds to $\D_{d-2}^{(I_{d-1}\smin{\set{i_{j-1}}})}.$
\end{lem}

\begin{proof}
 To take account of the `additional' faces $N_m$ of $\cl{\cube_{n\phantom{1}}^{(I^\prime_n)}}$ produced during the blow-up transformations, we carry out the full blow-up transformation by proposition~\ref{prop_blowup_trans_total}, yielding
\begin{gather}
\tilde{p}^{i_1}\ce p^{i_1}+\dotsc+p^{i_n},\\
\tilde{p}^{i_2}\ce
\begin{cases}
\frac{p^{i_2}+\dotsc+p^{i_n}}{p^{i_1}+p^{i_2}+\dotsc+p^{i_n}}&\text{for ${p^{i_1}+\dotsc+p^{i_n}}>0$}\\
 0 	 &\text{for ${p^{i_1}+\dotsc+p^{i_n}}=0$,}
\end{cases}\\
\hspace{0.6em}\vdots\notag\\
\tilde{p}^{i_j}\ce
\begin{cases}
\frac{p^{i_{j}}+\dotsc+p^{i_n}}{p^{i_{j-1}}+p^{i_{j}}+\dotsc+p^{i_n}}&\text{for $p^{i_{j-1}}+\dotsc+p^{i_n}>0$}\\
 0 	 &\text{for $p^{i_{j-1}}+\dotsc+p^{i_n}=0$,}
\end{cases}\\
\hspace{0.6em}\vdots\notag\\
\tilde{p}^{i_{n}}\ce
\begin{cases}
\frac{p^{i_n}}{p^{i_{n-1}}+p^{i_n}}&\text{for $p^{i_{n-1}}+p^{i_n}>0$}\\
 0 	 &\text{for $p^{i_{n-1}}+p^{i_n}=0$}
\end{cases}
\end{gather}
for $p\in\bigcup_{d=0}^n\D_{d\phantom{1}}^{(I_d)}$ 
and conversely
\begin{gather}
{p}^{i_1}=\tilde{p}^{i_1}(1-\tilde{p}^{i_2}),\\
\hspace{0.6em}\vdots\notag\\\label{eq_blowup_rev_j}
{p}^{i_j}=\tilde{p}^{i_1}\cdots\tilde{p}^{i_j}(1-\tilde{p}^{i_{j+1}}),\\
\hspace{0.6em}\vdots\notag\\
{p}^{i_{n-1}}=\tilde{p}^{i_1}\dotsm\tilde{p}^{i_{n-1}}(1-\tilde{p}^{i_{n}}),\\
{p}^{i_n}= \tilde{p}^{i_1}\dotsm\tilde{p}^{i_n}
\end{gather}
for $\tilde{p}\in\bigcup_{d=0}^n\cube_{d\phantom{1}}^{(I_d^\prime)}$ (note that we also have
${p}^{i_0}=1-\tilde{p}^{i_1}$); however, the given equations also smoothly extend to the entire $\cl{\cube_{n\phantom{1}}^{(I^\prime_n)}}$. This allows it to also transform the $N_d\subset\overline{\cube}_n$ back to $\overline{\D}_n$, \ie $\tilde{p}^{i_d}=0$ implies $p^{i_d},\dotsc,p^{i_n}=0$, whereas $0<\tilde{p}^{i_1},\dotsc,\tilde{p}^{i_{d-1}}<1$ leads to $p^{i_1},\dotsc,p^{i_{d-1}}>0$. Keeping the values of $\tilde{p}^{i_{d+1}},\dotsc,\tilde{p}^{i_n}$ yields the pivotal allele (limit) ratios  $\frac{p^{i_{d+1}}+\dotsc+p^{i_{n}}}{p^{i_{d}}+p^{i_{d+1}}+\dotsc+p^{i_{n}}},\dotsc,\frac{p^{i_n}}{p^{i_{n-1}}+p^{i_n}}$. If however $\tilde{p}^{i_j}=1$, this corresponds to ${p}^{i_{j-1}}=0$  (and $p^{i_{1}},\dotsc,{p}^{i_{j-1}},{p}^{i_{j+1}}\dotsc,p^{i_{d}}>0$ if $0<\tilde{p}^{i_1},\dotsc,\tilde{p}^{i_{j-1}},\tilde{p}^{i_{j+1}},\dotsc,\tilde{p}^{i_d} <1$ and $\tilde{p}^{i_{d+1}}=0$).
\end{proof} 

\begin{proof}[Proof of lemma~\ref{lem_blowup_reg}]
By lemma~\ref{lem_blowup_trans} and proposition~\ref{prop_blowup_trans_total} \resp corollary~\ref{cor_blowup_solution}, the full {blow-up transformation}\index{blow-up transformation!iterated} respectively maps 
\begin{align}\label{eq_blowup_assemb}
\bigcup_{d=0}^n\D_d^{(I_d)}\longmapsto\bigcup_{d=0}^n\cube_{d}^{(I^\prime_d)}
\end{align}
$C^\infty$-diffeomorphically (\cf equation~\eqref{eq_diffeom_dimwise_0}). By the $C_{p_0}^\infty$-regu\-lar\-ity of~$U$,  $u_n$ in $\D_n^{(I_n)}$ smoothly connects with  $u_{n-1}$ in $\D_{n-1}^{(I_{n-1})}$, and consequently so does $\tilde{u}_n$ in $\cube_n^{(I^{\prime}_n)}$ with $\tilde{u}_{n-1}$ in $\cube_{n-1}^{(I^{\prime}_{n-1})}$; an analogous statement holds for all lower dimensions.
Thus it remains to be shown that $\tilde{U}$ extends those faces of $\cl{\cube_n^{(I^\prime_n)}}$ given by $\set{\tilde{p}^{i_j}=1}$ for $j=1,\dotsc,n$ such that the extension is of class $C_p^\infty$.

In anticipation of lemma~\ref{lem_seq_N} on p.~\pageref{lem_seq_N}, the interior of $\set{\tilde{p}^{i_j}=1}\subset\cl{\cube_n^{(I^\prime_n)}}$ corresponds to $p^{i_{j-1}}=0$ and $p^{i_l}>0$ for $l\neq j-1$ in $\cl{\D^{(I_n)}_n}$, thus to $\D^{(I_n\smin\set{i_{j-1}})}_{n-1}$, which is a boundary face of $\D^{(I_n)}_n$ outside the assumed extension path defined by the (ordered) $I_n$. Hence by the $C^\infty_{p_0}$-regu\-lar\-ity, the relevant continuous extension of~$U$ needs to be zero there, and this is attained smoothly when coming from the interior $\D_n^{(I_n)}$. Considering the diffeomorphism properties of the transformation, this also applies to the cube.

An analogous observation holds for subcubes ${\cube_{d-1}^{(I^\prime_{d-1})}}\subset\cl{\cube}_n$, $d=1,\dots,n$: The interior of its face $\set{\tilde{p}^{i_j}=1}$ corresponds to $\D_{d^\prime-1}^{(I_{d-1}\smin\set{i_{j-1}})}\subset\cl{\D_{d-1}^{(I_{d-1})}}$ when transformed back to the simplex (\cf equation~\eqref{eq_blowup_assemb} and lemma~\ref{lem_seq_N}). This is again outside the assumed \txind{extension path}, in particular if starting in ${\D_{d-1}^{(I_{d-1})}}$, and hence the corresponding boundary extension of $u_{d-1}$ needs to smoothly attain zero there by the $C^\infty_{p_0}$-regu\-lar\-ity, which likewise applies analogously to the cube. 
\end{proof}

\section{The uniqueness of solutions of the stationary \KBE}\label{sec_uni}


The main application of the blow-up scheme is the uniqueness proof for the iteratively extended solutions of the \KBE in accordance with the extension constraints\index{extension constraints (n-dim)@extension constraints ($n$-dim)}~\ref{dfi_ext}. However, as already mentioned, in the presented work, this is limited to the stationary components. First, we will discuss the uniqueness of solutions of the correspondingly transformed stationary \KBE \index{Kolmob@\KBE!stationary ($n$-dim)}on the cube (which is basically analogous to the simplex, \cf section 10 in \cite{THJ5}). After that, the main result will be stated by applying the uniqueness result for the cube to the transformed iteratively\index{backward extension!iterated} extended solutions (assuming sufficient regularity if necessary).

Regarding the uniqueness of stationary solutions on the cube with the transformed Kolmogorov backward operator\index{backward operator} given by equation~\eqref{eq_L_cube_0}, we have in conjunction to lemma~10.1 
in \cite{THJ5} for the simplex: 

\begin{lem}[stem lemma, cube version]\label{lem_stem_cube}
For a solution $u\in C^\infty(\cube_n)$ of the stationary \KBE\index{Kolmob@\KBE!stationary ($n$-dim)} $\tilde{L}_n^* u=0$ 
\text{in $\cube_{n}$} 
with
\begin{align}\label{eq_L_cube_stat}
\tilde{L}_n^*\ce
\half\sum_{i=1}^{n} \frac{\tilde{p}^{i}(1-\tilde{p}^{i})}{\prod^{i-1}_{j=1} \tilde{p}^{j}}\ddsq{}{\tilde{p}^{i}}
\end{align}
and with extension $U\in C_p^\infty(\cl{\cube}_n)$, 
we have 
\begin{align}\label{eq_L_cube_stat_ext}
\tilde{L}^*U=0\quad\text{in $\cl{\cube}_n$},
\end{align}
\ie 
\begin{align}\label{eq_L_cube_stat_d-face}
\tilde{L}_d^*U=0\quad\text{with}\quad
\tilde{L}_d^*\ce
\half\sum_{\substack{i=\maxind(d)+1\\i \neq i_m}}^{n} \frac{\tilde{p}^{i}(1-\tilde{p}^{i})}{\prod\limits^{i-1}_{\substack{{j=\maxind(d)+1}\\j\neq i_m}} \tilde{p}^{j}}\ddsq{}{\tilde{p}^{i}}
\end{align}
in $\cube_d=\bigset{\tilde{p}^{i_1}=b_{i_1},\dotsc,\tilde{p}^{i_{n-d}}=b_{i_{n-d}}}\subset\bd_{d}\cube_n$ for all $1\leq d\leq n-1$ and all $i_1,\dotsc,i_{n-d}\in\set{1,\dotsc,n}$, $i_k\neq i_l$ for $k\neq l$ with $\maxind=\maxind(d)\ce\underset{i_1,\dotsc,i_{n-d}}{\argmax}\set{b_{i_m}=0}$ \resp $\maxind(d)\ce0$ if $b_{i_m}=1$ for all $i_m$.
\end{lem}

\begin{proof}
The statement is proven iteratively: Assuming that equation~\eqref{eq_L_cube_stat_d-face} holds in some (arbitrary) domain $\cube_{d+1}\subset\bd_{d+1}\cube_n$, we show that a corresponding formula also holds for any $\cube_{d}\subset\bd_{d}\cube_{d+1}\subset\bd_d\cube_n$. A repeated application of the argument then yields the assertion.

Let $\cube_{d+1}=\bigset{\tilde{p}^{i_1}=b_1,\dotsc,\tilde{p}^{i_{n-d-1}}=b_{n-d-1}}$ and $\cube_{d}=\bigset{\tilde{p}^{i_1}=b_1,\dotsc,\tilde{p}^{i_{n-d}}=b_{n-d}}$ with $i_{n-d}\neq i_1,\dotsc,i_{n-d-1}$ and $b_{n-d}\in\set{0,1}$. If we have $i_{n-d}<\maxind(d+1)$, then as $\tilde{p}^{i_{n-d}}\to 0$ \resp $\tilde{p}^{i_{n-d}}\to 1$, the value of the operator in equation~\eqref{eq_L_cube_stat_d-face} applied to $U$ -- with the occurring derivatives and the coefficients being continuous -- depends continuously on~$\tilde{p}$ up to the boundary, thus equation~\eqref{eq_L_cube_stat_d-face}, which already has the corresponding form for $\cube_{d}$ (note $\maxind(d)\equiv\maxind(d+1)$), also holds on $\cube_{d}$.

If we rather have $ i_{n-d}>\maxind(d+1)$ and $b_{n-d}=1$, then, when choosing some $\tilde{p}\in\cube_{d}$ and a sequence $(\tilde{p}_l)_{l\in\N}$ in $\cube_{d+1}$ with $\tilde{p}_l\to \tilde{p}$, the expression
\begin{align}
\half\frac{\tilde{p}_l^{i_{n-d}}(1-\tilde{p}_l^{i_{n-d}})}{\prod_{\substack{j={\maxind(d)}+1\\j\neq i_m}}^{i_{n-d}-1} \tilde{p}_l^{j}}\ddsq{}{\tilde{p}_l^{i_{n-d}}}U(\tilde{p}_l)
\end{align}
is controlled by $(1-\tilde{p}_l^{i_{n-d}})$ while approaching~$\tilde{p}$ and -- with the derivatives of $U$ being bounded on a closed neighbourhood of~$\tilde{p}$ by reason of the regularity of~$U$ -- is continuous up to~$\tilde{p}$. Analogous to the previous case, all other summands of the operator in equation~\eqref{eq_L_cube_stat_d-face} are also continuous on the boundary, thus proving that the corresponding form of equation~\eqref{eq_L_cube_stat_d-face} (with the $i_{n-d}$\ord summand deleted) holds in $\cube_{d}$  (again $\maxind(d)\equiv\maxind(d+1)$).

If instead $i_{n-d}>\maxind(d+1)$ and $b_{n-d}=0$, then we may multiply the whole equation~\eqref{eq_L_cube_stat_d-face} by $\tilde{p}^{i_{n-d}}$. If now $\tilde{p}^{i_{n-d}}\to 0$, then by a similar argument as above all derivatives of the operator that do not contain $\tilde{p}^{i_{n-d}}$ in the denominator of their coefficient continuously vanish, whereas the  values of all other summands are also continuous up to the boundary. Thus, equation~\eqref{eq_L_cube_stat_d-face} holds  on $\cube_{d}$ with the index $\maxind(d+1)$ replaced by $\maxind(d)=i_{n-d}$.
\end{proof}

The obtained equation~\eqref{eq_L_cube_stat_ext} may again be perceived as an extended version of the stationary \KBE on the cube (\cf also equation~\eqref{eq_back_n_stat_ext}, although the domains do not fully correspond), and we have (\cf proposition~10.2 
in \cite{THJ5}):
\begin{prop}\label{prop_unique_cube}
A solution $U\in C_p^\infty(\cl{\cube}_n)\cap C^0\big(\cl{\cube}_{n}\big)$ of %
the extended stationary \KBE\index{Kolmob@\KBE!extended stationary ($n$-dim)}
\begin{align}\label{eq_L_cube_simple_2}
\tilde{L}^*U=0 \quad\text{in $\cl{\cube}_{n}$}
\end{align}
with $\tilde{L}^*$ as in equation~\eqref{eq_L_cube_stat_d-face} 
%
is uniquely determined by its values on $\bd_0\cube_{n}$.
\end{prop}

\begin{proof}
The uniqueness may be shown by a successive application of the maximum principle: In every instance of the domain ${\cube_{d}}\subset\bd_d\cube_{n}$ for all $1\leq d \leq n$,  
the solution $U\vert_{{\cube_{d}}}$ is uniquely defined by its values on $\bd{\cube_{d}}$: If equation~\eqref{eq_L_cube_stat_d-face} comprises~$d$ derivative terms, this follows directly from Hopf's maximum principle as the operator is locally uniformly elliptic on $\cube_{d}$; if it only comprises $d^\prime<d$ derivative terms, analogous considerations apply for each $d^\prime$-dimen\-sional fibre of ${\cube_{d}}$ (with corresponding boundary part), thus giving the uniqueness of a solution on every fibre first and after assembling also on all $\cube_{d}$.
Applying this consideration successively for $\bd_{0}\cube_{n},\dotsc,\bd_{n}\cube_{n}=\cube_{n}$ yields the desired global uniqueness.
\end{proof}

With the blow-up scheme at hand, the preceding uniqueness result may also be conveyed to the simplex $\cl{\D}_n$, assuming some additional regularity. We eventually have: 

\begin{thm}\label{thm_uni_backw}
Let $n\in\Np$, $I_d\ce\set{i_0,i_{1},\dotsc,i_d}\subset \set{0,1,\dotsc,n}$ for $d=0,\dotsc,n$ with $i_i\neq i_j$ for $i\neq j$ and $u_{\set{i_0}}\co \D_0^{(\set{i_0})}\map \R$ be given. Then an extension $\bar{U}_{\set{i_0}}^{i_0,\dotsc,i_{n}}\co\bigcup_{0\leq d\leq n}\D_d^{(I_d)}\map\R$ 
 as in proposition~\ref{prop_ext_iter} is unique within the class of extensions~$U$ which satisfy the extension constraints\index{extension constraints (n-dim)@extension constraints ($n$-dim)}~\ref{dfi_ext}, \ie
\begin{itemize}
 \item[(i)]{are of class $C_{p_0}^\infty\big(\bigcup_{0\leq d\leq n}\D_d^{(I_d)}\big)$ with $U|_{\D^{(\set{i_0})}_0}=u_{\set{i_0}}$ and}
 \item[(ii)]{solve the stationary \KBE\eqref{eq_back_n_stat_ext}
 in $\bigcup_{0\leq d\leq n}\D_d^{(I_d)}$,}
 \end{itemize}
as well as, in case $n\geq2$, whose
 \begin{itemize}
 \item[(iii)]{transformation image $\tilde{U}\co\bigcup_{d=0}^n\cube_{d}^{(I^\prime_d)}\map\R$ by a successive blow-up transformation $\Phi^{\si_{n-1}}_{\ri_{n-1}}\circ\dotso\circ\Phi^{\si_1}_{\ri_1}$ as in proposition~\ref{prop_blowup_trans_total} has an extension to the entire boundary $\bd{\cube_{n\phantom{1}}^{(I_n^\prime)}}$ which is of class  $C_p^\infty\big(\cl{\cube_{n\phantom{1}}^{(I_n^\prime)}}\big)
 \cap C^0\big(\cl{\cube_{n\phantom{1}}^{(I_n^\prime)}}\big)$.
 }  
\end{itemize}
Consequently, also the global extension $\bar{U}_{\set{i_0}}$ as in proposition~8.4 
in \cite{THJ5}
\resp also in theorem~\ref{thm_sol_back_n_ext} is unique.
\end{thm}

\begin{proof}
The assertion for the trivial case $n=1$ directly follows, as $\bar{U}_{\set{i_0}}^{i_0,i_1}$ is as already sufficiently regular in $\cl{\D_1^{(I_1)}}\equiv\cl{\cube_1^{(I^\prime_1)}}$ for an application of the maximum principle, in particular globally continuous.
For $n\geq2$, any function~$U$ which is 
a solution of the stationary \KBE~\eqref{eq_back_n_stat_ext} in $\cl{\D^{(I_n)}_n}$  
by a full {blow-up transformation}\index{blow-up transformation!iterated} of the domain transforms into a function $\tilde{U}$, which solves the stationary \KBE~\eqref{eq_L_cube} in $\bigcup_{d=0}^n\cube_{d}^{(I^\prime_d)}$ (\cf proposition~\ref{prop_blowup_trans_total} \resp corollary~\ref{cor_blowup_solution} and lemma~\ref{lem_blowup_reg}). Furthermore, with the assumed regularity after a full blow-up, it has an extension to $\cl{\cube_{n\phantom{1}}^{(I_n^\prime)}}$ which is pathwise smooth as well as globally continuous and by lemma~\ref{lem_stem_cube} solves the stationary \KBE $\tilde{L}^*\bar{\tilde{U}}=0$ in $\cl{\cube_{n\phantom{1}}^{(I_n^\prime)}}$ 
with $\tilde{L}^*$ as in equation~\eqref{eq_L_cube_stat_d-face}. Hence, the uniqueness result of proposition~\ref{prop_unique_cube} applies and proves the uniqueness of the transformed function (and, regarding the injectivity of the blow-up, also the uniqueness of $U$) -- for specified boundary data on the entire  $\bd_0\cube_{n\phantom{1}}^{(I_n^\prime)}$. Thus, we only need to show that this boundary data is uniquely determined by the assumptions made. 

This is straightforward: In accordance with lemma~\ref{lem_blowup_reg}, $\tilde{U}$ \resp its corresponding continuous extension vanishes on  $\set{\tilde{p}^{i_j}=1}\subset\bd\cube_{n\phantom{1}}^{(I_n^\prime)}$, $j=1,\dotsc,n$. As by assumption~(iii) the continuous extendability applies to the entire $\cl{\cube_{n\phantom{1}}^{(I_n^\prime)}}$, $\tilde{U}$ \resp its extension even vanishes on
\begin{align}
\cl{\bigset{\tilde{p}^{i_1}=1}},\dotsc,
\cl{\bigset{\tilde{p}^{i_n}=1}}.
\end{align}
In particular, this signifies that $\tilde{U}$ \resp its extension vanishes on any vertex $\cube_0\subset\bd_0\cube_{n\phantom{1}}^{(I_n^\prime)}$ -- which may always be written as
\begin{align}
\cube_0=\bigset{\tilde{p}^{i_j}=b_j\quad\text{for $j=1,\dotsc,n$}}\quad\text{with correspondingly $b_j\in\set{0,1}$ --} 
\end{align}
except for the vertex $\cube_0^{(\varnothing)}=\set{(0,\dotso,0)}$, where it attains the value $u_{\set{i_0}}$ as stated previously. Thus, the (transformed) boundary data given on all vertices is identical for any extension in question, and since $\bar{U}_{\set{i_0}}^{i_0,\dotsc,i_{n}}\co\bigcup_{0\leq d\leq n}\D_d^{(I_d)}\map\R$ 
as in proposition~\ref{prop_ext_iter} satisfies the extension constraints\index{extension constraints (n-dim)@extension constraints ($n$-dim)} and has an extension to the entire boundary $\bd{\cube_{n\phantom{1}}^{(I_n^\prime)}}$ which is in $C_p^\infty(\cl{\cube_{n\phantom{1}}^{(I_n^\prime)}})\cap C^0\big(\cl{\cube_{n\phantom{1}}^{(I_n^\prime)}}\big)$ (this may be seen directly from equation~\eqref{eq_blowup_sol_sim_0}), it also is the unique extension.
\end{proof}


\begin{thebibliography}{1}

\bibitem{amari}{S.~Amari, H.~Nagaoka}, Methods of information geometry, Translations of mathematical monographs; v. 191, American Mathematical Society, 2000

\bi{ajls} N.~Ay, J.~Jost, H.\,V.~L\^e, L.~Schwachhöfer, Information geometry, monograph, to appear

\bibitem{BBMcK}
{G.\,J.} Baxter, {R.\,A.} Blythe, {A.\,J.} McKane.
\newblock {Exact Solution of the Multi-Allelic Diffusion Model}.
\newblock {\em Mathematical Biosciences}, 209(1):124--170, 2007.


\bibitem{buerger} R.~B\"urger, The mathematical theory of selection,
  recombination, and mutation, John Wiley, 2000

\bi{epstein1} C.\,L.~Epstein, R.~Mazzeo, Wright-Fisher diffusion in one
dimension, SIAM J. Math. Anal. 42 (2010), 568--608

\bi{epstein2} C.\,L.~Epstein, R.~Mazzeo, Degenerate diffusion operators
arising in population biology, Princeton Univ. Press, 2013

\bi{ethier1}S.\,N.~Ethier, A class of degenerate diffusion processes
occurring in population genetics, Comm. Pure Appl. Math. 29 (1976), 483--493


\bibitem{EG1993}
S.\,N.~Ethier, {R.\,C.} Griffiths, The Transition Function of a Fleming-Viot Process, Ann Probab 21 (1993) 1571-1590.
 

\bi{ethier2}S.\,N.~Ethier, T.~Kurtz, Markov processes: characterization
and convergence, Wiley, 1986, 2005


\bi{ethier3}S.\,N.~Ethier, T.~Nagylaki, Diffusion Approximations of
Markov Chains with Two Time Scales and Applications to Population
Genetics, Advances in Applied Probability 12 (1980), 14--49
  


\bibitem{ewens}
W.\,J.~Ewens, {\it Mathematical Population Genetics I. Theoretical Introduction}, Springer-Verlag New York Inc., Interdisciplinary Applied Mathematics, 2nd ed., 2004.

\bibitem{fee}
P.\,M.\,N.~Feehan, \textit{Maximum principles for boundary-degenerate linear parabolic differential operators}, arXiv preprint, arXiv:1306.5197 (2013).


\bibitem{fisher}
R.\,A.~Fisher, {\it On the dominance ratio}, Proc. Roy. Soc. Edinb.,
{\bf 42} (1922), 321-341.

\bibitem{Gri1979}
{R.\,C.} Griffiths, A Transition Density Expansion for a Multi-Allele Diffusion Model, Adv Appl Probab 11 (1979) 310-325

\bibitem{Gri1980}
{R.\,C.} Griffiths, Lines of descent in the diffusion approximation of neutral Wright-Fisher models. Theoret Popn Biol 17 (1980) 37–50

\bibitem{GS2010}
{R.\,C.} Griffiths, D. Span\'o , Diffusion processes and coalescent trees,
Chapter 15, pp. 358-375. In: Probability and Mathematical Genetics, Papers in Honour of Sir John Kingman. LMS Lecture Note Series 378. ed
Bingham, N. H. and Goldie, C. M., Cambridge University Press, 2010

\bi{julian} J.~Hofrichter, On the diffusion approximation of Wright-Fisher models with several alleles and loci and its geometry, PhD thesis, University of Leipzig, 2014
   
\bibitem{THJ4} J.~Hofrichter, {T.\,D.} Tran, J.~Jost. A hierarchical extension scheme for solutions of the Wright–Fisher model, submitted, arXiv:1406.5152
  
\bibitem{THJ5} J.~Hofrichter, {T.\,D.} Tran, J.~Jost. A hierarchical extension scheme for backward solutions of the Wright–Fisher model, submitted, arXiv:1406.5146
  
  
\bibitem{jost_pde} J.~Jost.
\newblock {\em {Partial Differential Equations}}, volume 214 of {\em {Graduate
  Texts in Mathematics}}.
\newblock Springer, Berlin, Heidelberg, 3{rd} edition, 2013.

\bibitem{jost_bio} J.~Jost.
\newblock {\em Mathematical methods in biology and neurobiology},
Universitext, \newblock Springer, London, Heidelberg, 2014
\bi{karlin} S.~Karlin, H.\,M.~Taylor, A second course in stochastic
processes, Academic Press, 1981
 

\bibitem{kimura1}
M. Kimura, {\it Solution of a Process of Random Genetic Drift with a
Continuous Model}, PNAS--USA, Vol. 41, No. 3, (1955), 144-150.

\bibitem{kimura2}
M. Kimura, {\it Random genetic drift in multi-allele locus}, Evolution, \textbf{9} (1955), 419-435.

\bibitem{kimura3}
M. Kimura, {\it Random genetic drift in a tri-allelic locus; exact solution with a continuous model}, Biometrics, \textbf{12} (1956), 57-66.


\bibitem{kimura_3all}
M.~Kimura.
\newblock {Random Genetic Drift in a Tri-Allelic Locus; Exact Solution with a
  Continuous Model}.
\newblock {\em Biometrics}, 12(1):57--66, 1956.

\bibitem{King1982}
J.\,F.\,C.~Kingman, The coalescent, Stoch Proc Appl 13 (1982) 235-248



\bibitem{littler_loss}
{R.\,A.} Littler.
\newblock {Loss of Variability at One Locus in a Finite Population}.
\newblock {\em Mathematical Biosciences}, 25(1-2):151--163, 1975.

\bibitem{lit-fack}
{R.\,A.} Littler and {E.\,D.} Fackerell.
\newblock {Transition Densities for Neutral Multi-Allele Diffusion Models}.
\newblock {\em Biometrics}, 31(1):117--123, 1975.

\bibitem{littler-good_ages}
{R.\,A.} Littler and {A.\,J.} Good.
\newblock {Ages, Extinction Times, and First Passage Probabilities for a
  Multiallele Diffusion Model with Irreversible Mutation}.
\newblock {\em Theoretical Population Biology}, 13(2):214--225, 1978.

\bibitem{oks_sde}
B.~{\O}ksendal.
\newblock {\em {Stochastic Differential Equations}}.
\newblock {Universitext}. Springer, Berlin, Heidelberg, 6{th} edition, 2003.


\bibitem{Sato1978}
K. Sato, Diffusion operators in population genetics and convergence of Markov chains, In Measure Theory Application to Stochastic Analysis, LNM 695 (1978) 127-137

\bibitem{Shi1977}
N. Shimakura, \'Equations diff\'erentielles provenant de la g\'en\'etique des populations, Tohoku Math J. 29 (1977) 287-318 

\bibitem{Shi1981}
N. Shimakura, Formulas for diffusion approximations of some gene frequency models, J Math Kyoto Univ 21 (1981) 19-45


\bibitem{Tav1984}
Tavar\'e S., Line-of-descent and genealogical processes, and theor applications in population genetics models, Theor Popn Biol 26 (1984) 119-164


\bibitem{Dat}
{T.\,D.} Tran.
\newblock {\em {Information Geometry and the Wright-Fisher Model of
  Mathematical Population Genetics}}.
\newblock PhD thesis, University of Leipzig, 2012.


\bibitem{THJ1} {T.\,D.} Tran, J.~Hofrichter, J.~Jost. An introduction to the mathematical structure of the Wright-Fisher model of
  population genetics, Theory Biosc. (2012), 1--10

\bibitem{THJ2} {T.\,D.} Tran, J.~Hofrichter, J.~Jost. A general solution
  of the Wright-Fisher model of random genetic drift, submitted, arXiv:1207.6623 

\bibitem{THJ3} {T.\,D.} Tran, J.~Hofrichter, J.~Jost. The evolution
  of moment generating functions for the Wright-Fisher model of
  population genetics, submitted,  arXiv:1401.5219

%
\bibitem{wright1}
S.~Wright, {\it Evolution in Mendelian populations}, Genetics, {\bf 16} (1931), 97-159.


\bibitem{wright2}
S.~Wright, Adaptation and selection. In: G.Jepson, E.Mayr, G.Simpson
(eds.), Genetics, Paleontology, and Evolution, pp.365-389, Princeton
Univ.Press, 1949
\end{thebibliography}

\end{document}